\documentclass{amsart}
\usepackage{amsmath,amsfonts,amsthm, amssymb}
\usepackage{hyperref}
\usepackage{color}
\usepackage[font=footnotesize]{idxlayout}
\newtheorem{lemma}[equation]{Lemma}
\newtheorem{proposition}[equation]{Proposition}
\newtheorem{theorem}[equation]{Theorem}
\newtheorem{corollary}[equation]{Corollary}
\numberwithin{equation}{section}
\theoremstyle{definition}
\newtheorem{definition}[equation]{Definition}
\newtheorem{context}[equation]{Context}
\newtheorem{notation}[equation]{Notation}
\theoremstyle{remark}
\newtheorem{example}[equation]{Example}
\newtheorem{remark}[equation]{Remark}

\newcommand\Ga{{\mathbb G_a}}
\newcommand\bB{{\mathbf B}}
\newcommand\bG{{\mathbf G}}
\newcommand\bH{{\mathbf H}}
\newcommand\bK{{\mathbf K}}
\newcommand\bL{{\mathbf L}}
\newcommand\bQ{{\mathbf Q}}
\newcommand\bS{{\mathbf S}}
\newcommand\bT{{\mathbf T}}
\newcommand\bU{{\mathbf U}}
\newcommand\BQ{{\mathbb Q}}
\newcommand\BZ{{\mathbb Z}}
\newcommand\CA{{\mathcal A}}
\newcommand\CC{{\mathcal C}}
\newcommand\CO{{\mathcal O}}
\newcommand\Gun{{\bG^1}}
\newcommand\Tun{{\bT^1}}
\newcommand\Gtso{{(\bG^{t\sigma})^0}}
\newcommand\Gts{{\bG^{t\sigma}}}
\newcommand\Gs{{\bG^\sigma}}
\newcommand\Gso{{(\Gs)^0}}
\newcommand\Lso{{(\bL^\sigma)^0}}
\newcommand\Tso{{(\bT^\sigma)^0}}
\newcommand\Zso{{(Z(\bG)^\sigma)^0}}
\newcommand\Ts{{\bT^\sigma}}
\newcommand\LT{{\bT/[\bT,\sigma]}}
\newcommand\tW{{\widetilde W}}
\newcommand\ve{{\varepsilon}}
\newcommand\Wa{W_a}
\newcommand\lexp[2]{\kern\scriptspace\vphantom{#2}^{#1}\kern-\scriptspace#2}
\newcommand\inv{^{-1}}
\newcommand\ptors{{p'\text{-}\mathrm{tors}}}
\DeclareMathOperator\Aut{\mathrm{Aut}}
\DeclareMathOperator\AZ{\mathrm{AZ}}
\DeclareMathOperator\Hom{\mathrm{Hom}}
\DeclareMathOperator\Ker{\mathrm{Ker}}
\DeclareMathOperator\Rad{\mathrm{Rad}}
\DeclareMathOperator\Res{\mathrm{Res}}
\newcommand\GL{\mathrm{GL}}
\newcommand\SL{\mathrm{SL}}
\newcommand\PGL{\mathrm{PGL}}
\newcommand\Orth{\mathrm{O}}
\newcommand\PSO{\mathrm{PSO}}
\newcommand\SO{\mathrm{SO}}
\newcommand\Sp{\mathrm{Sp}}
\newcommand\PSp{\mathrm{PSp}}
\newcommand\Spin{\mathrm{Spin}}
\newcommand\Id{\mathrm{Id}}
\newcommand\ad{\mathrm{ad}}
\newcommand\car{\mathrm{char}}
\newcommand\diag{\mathrm{diag}}
\newcommand\antidiag{\mathrm{antidiag}}
\newcommand\qss{quasi-semisimple}
\newcommand\card[1]{{|#1|}}
\newcommand\genby[1]{\mathopen\langle#1\mathclose\rangle}
\newcommand\pairing[2]{{\mathopen(#1,#2\mathclose)}}
\newcommand\CHEVIE{{\tt CHEVIE}}
\newcommand\GAP{{\tt GAP3}}
\newcommand{\nnode}[1]{{\kern 4.5pt\hbox to
0pt{\hss{$\mathop\bigcirc\limits_{#1}$}\hss}\kern 4.4pt}}
\newcommand{\sbar}{{\vrule width10pt height3pt depth-2pt}}
\newcommand{\dbar}{{\rlap{\vrule width10pt height2pt depth-1pt} 
                 \vrule width10pt height4pt depth-3pt}}
\newcommand\edge{{\vrule width10pt height3pt depth-2pt}}
\newcommand\vertbar[2]{\rlap{\kern4pt\vrule width1pt height17.3pt depth-7.3pt}
 \rlap{\raise19.4pt\hbox{$\kern -0.4pt\bigcirc\scriptstyle#2$}}
                 \nnode{#1}}
\title{Quasi-semisimple elements}
\author{F.~Digne and J.~Michel}
\address[J.~Michel]{Institut de Math\'ematiques de Jussieu -- Paris rive
gauche, Universit\'e Denis Diderot, B\^atiment Sophie Germain,
75013, Paris France.}
\email{jean.michel@imj-prg.fr}
\urladdr{webusers.imj-prg.fr/$\sim$jean.michel}
\address[F.~Digne]{Laboratoire Ami\'enois de Math\'ematique Fondamentale et 
Appliqu\'ee, CNRS UMR 7352, Universit\'e de Picardie-Jules Verne,
80039 Amiens Cedex France.}
\email{digne@u-picardie.fr}
\urladdr{www.lamfa.u-picardie.fr/digne}

\makeindex
\begin{document}
\maketitle
\begin{abstract}
We study \qss\ elements of disconnected reductive algebraic groups
over an algebraically closed field. We describe their centralizers, 
define isolated and quasi-isolated \qss\ elements and classify their
conjugacy classes.
\end{abstract}
\section*{Introduction}
In this paper we study conjugacy classes and centralizers
of  \qss\  elements  in  disconnected  reductive  groups over algebraically
closed fields.

Let  $\bG$ be an algebraic group  such that the connected component $\bG^0$
is  reductive.  Let  $\Gun$
\index{G1@$\bG$, $\Gun$}
be  a  connected  component  of  $\bG$
and $\bH$  be the  subgroup of  $\bG$ generated by $\Gun$.
Then  $\bG^0$-conjugacy  classes  of  $\Gun$  coincide  with $\bH$-conjugacy
classes, and $\bG$-conjugacy classes of $\Gun$ are unions of 
orbits of the  finite  group  $\bG/\bH$ on $\bH$-conjugacy classes.
Thus it is reasonable to first consider the case where $\bG$ is generated by
$\Gun$.

This is  closely related  to the  study of algebraic automorphisms of
$\bG^0$:  if $\bG^0$ is  adjoint, the elements of $\Gun$ can be interpreted
as a class of automorphisms of $\bG^0$ modulo inner automorphisms.

According  to  these  considerations,  in  what  follows  we  will  use the
following  setting: we  denote by  $\sigma$ an  element of a (disconnected)
reductive  group with identity
component denoted by $\bG$, and  we  study  the  conjugacy
classes  in the coset $\Gun:=\bG\cdot\sigma$.

We  consider algebraic groups  over some algebraically  closed field $k$ of
characteristic $p$ \index{p@$p$}, where $p$ is $0$ or a prime number.

We  denote  by $n_\sigma$ \index{nsigma@$n_\sigma$}
the smallest  power of  $\sigma$ which acts by an inner
automorphism on $\bG$.

A  {\em \qss} element is an element of  $\Gun$ which stabilizes a pair 
$\bT\subset\bB$ of a maximal torus of $\bG^0$ and a Borel subgroup containing
it;  let $\Tun$ \index{Tun@$\bT$, $\Tun$} be the normalizer
of  this pair  in $\Gun$.  
Since all such pairs are conjugate, and we are interested in \qss\ elements
up to conjugacy, we study the elements of $\Tun$.

An element  $\sigma\in\Tun$ acts  on $W:=N_\bG(\bT)/\bT$; \index{W@$W$}
its order as an automorphism of $\bT$ is $n_\sigma$. If $\bG$ is semisimple,
$n_\sigma$ is also the order of the action of $\sigma$ on $W$ since 
an algebraic automorphism of $\bG$ stabilizing $\bT$ and trivial on $W$
is inner.

Our results are as follows.

In  the  first  section,  we  study  the  centralizer  $\Gs$  of  $\sigma$,
generalizing  the  work  of  Steinberg  \cite{St}  when $\bG$ is semisimple
simply  connected,  and  of  \cite{grnc}  which  deals  with  the connected
component  $\Gso$. Then we  give a first  approach to the classification of
semisimple  classes,  reducing  the  problem  to questions on the quotient
torus $\LT$.

In  the  second  section,  we  fix  a  quasi-central  element $\sigma$ (see
Definition \ref{quasicentral}) to give more specific results: we need only to consider
elements  $t\sigma\in\Tso\cdot\sigma$,  and  we  introduce  a  root  system
$\Phi_\sigma$  such that  \qss\ classes  are classified  by elements of the
fundamental alcove of the corresponding affine Weyl group.

In the third section we introduce isolated and quasi-isolated elements and
indicate a method to classify them.
This allows us in particular to state Theorem \ref{A(ZGtso)}
which describes the exponent of the group $Z(\Gtso)/Z^0(\Gtso)$.

Finally in the last section we classify explicitly the quasi-isolated classes
in quasi-simple reductive groups, and describe their centralizers.

The assumption that $k$ is algebraically closed is unnecessary for quite a few
results in this paper, but we adopted it for simplicity.

We  thank Cedric Bonnaf\'e  twice. First, for  having written \cite{cedric}
which we generalize to our setting. Second, for giving us the
impulse  of writing this paper by asking us whether a result like Theorem \ref{A(ZGtso)}
could be true.
We thank Gunter Malle for his remarks on a first version of this paper.
\section{Quasi-semisimple elements}
\subsection*{The tori $\Tso$ and $\LT$.}
In this subsection, we describe facts valid for any torus
$\bT$ with an algebraic automorphism $\sigma$ of finite order $n$. 

Let $Y:=Y(\bT)=\Hom(k^\times,\bT)$ and $X:=X(\bT)=\Hom(\bT,k^\times)$,
\index{Y@$Y$}\index{X@$X$} and let $\pairing{}{}$ denote the natural
pairing between $Y$ and $X$.

We have $Y(\Tso)=Y^\sigma$.
To describe $X(\Tso)$ we introduce the map $\pi:X\otimes\BQ\to
X\otimes\BQ:x\mapsto n\inv\sum_{i=1}^n\sigma^i(x)$. \index{pi@$\pi$}
Then 
\begin{lemma}\label{X(Tso)} The restriction of
characters from $\bT$ to $\Tso$ induces an isomorphism 
$ \pi(X)\xrightarrow\sim X(\Tso)$.
The natural action of $W^\sigma$ on $X(\Tso)$ corresponds via this isomorphism
to the restriction of its action to $\pi(X)\subset X\otimes\BQ$.
\end{lemma}
\begin{proof}For $t\in\Tso$ and $\chi\in X$
all the elements in the $\sigma$-orbit of $\chi$ have same value on $t$, hence
the value $\chi(t)$ depends only on $\pi(\chi)$. 
Hence, if we show that $\ker\pi$ coincides with the characters which have 
a trivial restriction to
$\Tso$, we can factorize $\Res^\bT_\Tso$
through an injective 
$W^\sigma$-equivariant map  $\pi(X)\xrightarrow\iota X(\Tso)$.
The image of the map $N_\sigma:\bT\to\bT:t\mapsto
t\cdot\lexp\sigma t\ldots\lexp{\sigma^{n-1}}t$
is $\Tso$ since the image is a torus, consists of $\sigma$-stable elements, 
and $N_\sigma$ acts on $\Tso$ as raising to the $n$-th power. 
A character $\chi\in X$ is in
$\ker\pi$ if and only if $0=(1+\sigma+\ldots+\sigma^{n-1})\chi=
\chi\circ N_\sigma$ which is equivalent to $\Res^\bT_\Tso\chi=0$. 

It remains to show that $\iota$ is surjective; this comes from the
fact that any character of the subtorus $\Tso$ can be lifted to a character of
$\bT$, since a subtorus is a direct factor. 
\end{proof}

\begin{lemma}\label{A(Ts)}
\begin{enumerate}
\item $\Ts/\Tso$ is isomorphic to the $p'$-part of the group
$$H=\Ker(1+\sigma+\ldots+\sigma^{n-1}\mid X)/(\sigma-1)X,$$ 
whose exponent divides $n$.
\item If $\sigma$ permutes a basis of $X$, then
$\Ker(1+\sigma+\ldots+\sigma^{n-1}\mid X)=(\sigma-1)X$.
Thus if $\sigma$ permutes a basis of $X$ then $\Ts=\Tso$.
\item The intersection $[\bT,\sigma]\cap\Ts$ has exponent $n$, where
\index{Tz@$[\bT,\sigma]$}
$[\bT,\sigma]$ is the torus $\{t\sigma(t\inv)\mid t\in\bT\}$.
We have $\bT=[\bT,\sigma]\cdot\Tso$, and $Y([\bT,\sigma])=
\Ker(1+\sigma+\ldots+\sigma^{n-1}\mid Y)$.
\end{enumerate}
\end{lemma}
\begin{proof}
For  (1), $\Ts/\Tso$ is isomorphic  to $X(\Ts)_\ptors$, where $\ptors$ denotes
the $p'$-part of the torsion. Considering the map
$t\mapsto[t,\sigma]:\bT\to[\bT,\sigma]$  of  kernel  $\Ts$  and  using  the
exactness  of the functor  $X$, we get  that $X(\Ts)_\ptors$ is isomorphic to
$(X/(\sigma-1)X)_\ptors$. Since
$X/\Ker(1+\sigma+\ldots+\sigma^{n-1})$  has  no  torsion,  we get that
$X(\Ts)_\ptors$  is  isomorphic  to  $H_\ptors$.  
Now  modulo $\sigma-1$  we  have
$(1+\sigma+\ldots+\sigma^{n-1})x=nx$ whence the assertion that the exponent
of $H$ divides $n$.

(2) results from an explicit computation: let $x_1,\ldots,x_r$ be a basis
of $X$ such that $\sigma$ induces a permutation of
$\{1,\ldots,r\}$. Then $\sum_i{a_i x_i}\in\Ker(1+\sigma+\ldots+\sigma^{n-1})$
if and only if  for each orbit $\CO$ of $\sigma$ on $\{1,\ldots,r\}$ we
have $\sum_{i\in\CO}a_i=0$. On the other hand $\sum_i{a_i x_i}$ is in the
image of $\sigma-1$ if and only if $a_i$ is of the form $b_{\sigma\inv(i)}-b_i$
for some sequence $b_i$. The two conditions are clearly equivalent.

The first item of (3) results from the remark that if $[t,\sigma]\in\Ts$
then $[t,\sigma]^n=(t\sigma(t\inv))\sigma(t\sigma(t\inv))\ldots
\sigma^{n-1}(t\sigma(t\inv))=t\sigma^n(t\inv)$. The second item follows:
since the kernel of $t\mapsto[t,\sigma]:\bT\to[\bT,\sigma]$ is $\Ts$,
we have $\dim\bT=\dim\Tso+\dim[\bT,\sigma]$ and the intersection is finite by
the first part. Finally, for the third item, $Y([\bT,\sigma])$ is a pure
sublattice  of $Y$ which spans the same $\BQ$-subspace as the image of
$\sigma-1$; this is the case of $\Ker(1+\sigma+\ldots+\sigma^{n-1}\mid Y)$,
since a kernel is always a pure sublattice.
\end{proof}
When $\sigma$ permutes a basis of $X$, by
(2) of Lemma \ref{A(Ts)}, $\pi(X)$ coincides with the
$\sigma$-coinvariants of $X$. Though this assumption does not hold
in general, the duality between $Y^\sigma$ and $\pi(X)$ plays a similar role
to the duality between invariants and coinvariants.
To remind of this analogy we will use the following notation:
\begin{notation}
\index{Xsigma@$X_\sigma$}\index{Ysigma@$Y_\sigma$}
We write $X_\sigma$ for $\pi(X)$ and  
$Y_\sigma$ for $\pi(Y)$ where
$\pi:Y\to Y\otimes\BQ:y\mapsto n\inv\sum_{i=1}^n\sigma^i(y)$.
\index{pi@$\pi$}
\end{notation}

The following is true in general for lattices with a $\sigma$-stable pairing
\begin{lemma} \label{XsYs}
The  restriction of the duality  between $X\otimes\BQ$ and $Y\otimes\BQ$ to
$(X\otimes\BQ)^\sigma$  and $(Y\otimes\BQ)^\sigma$ is a duality which makes
the  lattices $(X_\sigma,Y^\sigma)$, as well as $(X^\sigma,Y_\sigma)$, dual
to each other.
\end{lemma}
\begin{proof}
That the restriction of the duality is a duality is clear. 

Let us compute the dual of $X_\sigma$.
If $y\in (Y\otimes\BQ)^\sigma$ and
$\pairing  y{\pi(x)}\in\BZ$ for all $x\in  X$, then, since $\pairing{}{}$ 
is $\sigma$-invariant we have $\pairing y{\pi(x)}=\pairing yx\in\BZ$  
for  all  $x\in X$, so that
$y\in Y$, hence $y\in Y^\sigma$. Conversely, if $y\in Y^\sigma$, then
$\pairing y{\pi(x)}\in\BZ$ for all $x\in  X$.

The proof for $(X^\sigma,Y_\sigma)$ is symmetric to the above proof.
\end{proof}

\begin{lemma}
Through the isomorphism  $(X(\Tso),Y(\Tso))\simeq(X_\sigma,Y^\sigma)$
(see Lemma \ref{X(Tso)}) the duality of Lemma \ref{XsYs}
between $X_\sigma$ and $Y^\sigma$ identifies with the
natural duality between $X(\Tso)$ and $Y(\Tso)$.
\end{lemma}
\begin{proof}
By definition, for $\lambda\in k$, $x\in X$ and $y\in Y$
we have $x(y(\lambda))=\lambda^{\pairing yx}$. 
If $y\in Y(\Tso)=Y^\sigma$ then
$x(y(\lambda))=(\Res^\bT_\Tso(x))(y(\lambda))$, which proves the lemma.
\end{proof}

Just as we identified $(X_\sigma,Y^\sigma)$ with $(X(\Tso),Y(\Tso))$, the 
following lemma allows to identify $(X^\sigma,Y_\sigma)$ with
$(X(\LT),Y(\LT))$.
\begin{lemma}\label{X(T/LT)}~
\begin{enumerate}
\item The morphism $\bT\to\LT$ induces an isomorphism $X(\LT)\simeq X^\sigma$.
\item
The natural map $f:Y\to Y(\LT)$ 
induces an isomorphism $Y_\sigma\simeq Y(\LT)$.
\item Through the above isomorphisms
the natural duality between $X(\LT)$ and $Y(\LT)$ identifies
with the duality of Lemma \ref{XsYs} between $X^\sigma$ and $Y_\sigma$.
\end{enumerate}
\end{lemma}
\begin{proof}
A character $x\in X$ is trivial on $[\bT,\sigma]$ if and
only if $x$ is $\sigma$-stable. Since $X$ is exact this gives
the isomorphism $X^\sigma\simeq X(\LT)$ of (1).

We  now prove (2).
The   image  $f(y)\in  Y(\LT)$   depends  only  on   $\pi(y)$:  indeed  let
$y\in\ker\pi$,  then $n_\sigma y$ is  in $\ker\pi$ so  for $\lambda\in k^\times$ we
have   $n_\sigma \pi(y)(\lambda)=1$,   that   is   $N_\sigma(y(\lambda))=1$, 
where $N_\sigma$ is as in the proof of Lemma \ref{X(Tso)}.
Thus $y(k^\times)$  is a  subtorus of $\ker
N_\sigma$  and  the  latter  contains  $[\bT,\sigma]$  as  a  finite  index
subgroup, thus   $y(k^\times)\subset[\bT,\sigma]$.   Conversely,   if
$y(k^\times)\subset[\bT,\sigma]$  then  $N_\sigma\circ  y$  is  trivial, so
$y\in\ker\pi$.  Hence  we  get  a  well  defined injective map $Y_\sigma\to
Y(\LT)$.  The  surjectivity  comes  from  the  fact that, being a subtorus $[\bT,\sigma]$
is a direct factor of $\bT$, so that a cocharacter of $\LT$ can be
lifted to $\bT$.

We  now prove (3). The pairing  between elements of $X(\LT)$ and
$Y(\LT)$  identified via (1) and (2)  respectively with $x\in X^\sigma$ and
$y\in   Y_\sigma$   is   given   by   the  computation  of  $\lambda\mapsto
x(y(\lambda)\text{mod}  [T,\sigma])= x(y(\lambda))=\lambda^{\pairing yx}$.
Since  $\pairing yx=\pairing{\pi(y)}x$  by
$\sigma$-invariance of the pairing and of $x$, we get (3).
\end{proof}
We note for future reference the following:
\begin{lemma}
$((\bT/\Tso)^\sigma)^0=\{1\}$.
\end{lemma}
\begin{proof}
By Lemma \ref{X(Tso)} the restriction map from $X(\bT)$ to $X(\Tso)$ is $\pi$. 
From the exactness of the functor $X$ we deduce $X(\bT/\Tso)=\ker\pi$. 
Applying Lemma \ref{X(Tso)} again we get 
$X(((\bT/\Tso)^\sigma)^0)=\pi(\ker(\pi))=\{0\}$, whence the result.
\end{proof}
\subsection*{Centralizers of \qss\ elements}
Let as in the introduction $\sigma\in\Gun$ be a \qss\ element stabilizing the
pair $\bT\subset\bB$ (thus $\sigma\in\Tun$).

\index{Sigma@$\Sigma$, $\Sigma^\vee$}
Let  $\Sigma$  be  the  root  system  of  $\bG$  with  respect to $\bT$ and
$\Sigma^\vee$  be  the  corresponding  set of coroots. The coroot corresponding to
$\alpha\in\Sigma$ will be denoted by $\alpha^\vee$.
The  $\sigma$-stable Borel
subgroup $\bB$ determines a $\sigma$-stable order on $\Sigma$ and we denote
by  $\Sigma^+$ and $\Pi$ respectively the positive and the simple roots for
this order.
\index{Pi@$\Pi$}
The element $\sigma$ induces a permutation of $\Pi$ and $\Sigma$.

We  recall  (see for instance \cite[1.8(v)]{grnc}) that,  
if  $x_\alpha:\Ga\xrightarrow\sim\bU_\alpha$  is  an 
isomorphism  to the root subgroup of  $\bG$ attached to $\alpha$, and $\CO$
is  the orbit under $\sigma$ of $\alpha$, there is a unique 
\index{Csigma@$C_{\sigma,\CO}$}
$C_{\sigma,\CO}\in k^\times$ (depending only on $\CO$ and not on $\alpha$)
such that $\lexp{\sigma^{|\CO|}}x_\alpha(\lambda)=
x_\alpha(C_{\sigma,\CO}\lambda)$   for   any   $\lambda\in   k$.

We introduce the following terminology:
\begin{definition}
\index{special,cospecial@special, cospecial}
We say that the $\sigma$-orbit $\CO\subset\Sigma$ is {\em special} if there
exist two roots $\alpha,\beta\in\CO$ such  that  $\alpha+\beta\in\Sigma$.
In this case, we say that the orbit of $\alpha+\beta$ is {\em cospecial}.
\end{definition}

\begin{remark}\label{roots orthogonal}~
Special orbits only exist if $\bG$ has a component of type $A_{2r}$.
If $\alpha$ and $\beta $ are distinct roots in the same $\sigma$-orbit then 
$\pairing{\alpha^\vee}\beta= 0$ except if the orbit is special and 
$\alpha+\beta$ is a root, in which case $\pairing{\alpha^\vee}\beta=-1$.
\end{remark}

\begin{definition}\label{Sigmatsigma}~
\index{alphabarcheck@$\bar\alpha^\vee$}\index{O(alpha)@$\CO(\alpha)$}
Let $\CO(\alpha)$ denote the orbit of $\alpha$ under $\sigma$.
\index{Sigmasigma@$\Sigma_\sigma$, $\Sigma_\sigma^\vee$, $\Sigma_{t\sigma}$, $\Sigma_{t\sigma}^\vee$}
\begin{itemize}
\item
If $p\ne 2$ we define
$\Sigma_\sigma:=\{\pi(\alpha)\mid\alpha\in\Sigma, C_{\sigma,\CO(\alpha)}=1\}
\subset X_\sigma$. If $p=2$ we define
$\Sigma_\sigma:=\{\pi(\alpha)\mid\alpha\in\Sigma, C_{\sigma,\CO(\alpha)}=1,
\text{$\CO(\alpha)$ not special}.\}$.
\item
We define
$\Sigma_\sigma^\vee:=
\{\bar\alpha^\vee\mid \pi(\alpha)\in\Sigma_\sigma, \CO(\alpha)\text{ not
special}\}\cup
\{2\bar\alpha^\vee\mid \pi(\alpha)\in\Sigma_\sigma, \CO(\alpha)\text{
special}\} \subset Y^\sigma$
where $\bar\alpha^\vee:=\sum_{\beta\in\CO(\alpha)}\beta^\vee$.
\end{itemize}
\end{definition}
\begin{proposition}\label{root system of Gtso}~
\begin{enumerate}
\item
The  group $\Gso$ is reductive;  its root system with  respect to $\Tso$ is
$\Sigma_\sigma$, and the corresponding set of coroots is
$\Sigma_\sigma^\vee$; we denote by $W^0(\sigma)$ the corresponding Weyl group.
\item
$W^\sigma$ acts on
$X_\sigma\subset X\otimes\BQ$ and on $Y^\sigma$ by restriction.
$(W^\sigma,S_\sigma)$ is a Coxeter system where 
$S_\sigma$ is the set of $w_\CO$ where $w_\CO$ denotes the longest
element of the parabolic subgroup $W_\CO$, and $\CO$ runs over the orbits of
$\sigma$ in $S$.
\item
\index{W0(sigma)@$W^0(\sigma)$, $W^0(t\sigma)$}
We have $N_\bG(\bT)^\sigma=N_\Gs(\Tso)$ and
$W^0(\sigma)\subset W^\sigma$ as subgroups of $\GL(X_\sigma)$.
\item
The set
$\Sigma_\sigma^+:=\{\pi(\alpha)\in\Sigma_\sigma\mid \alpha \in\Sigma^+\}$ 
\index{Sigmasigma+@$\Sigma_\sigma^+$}
is a set of positive roots in $\Sigma_\sigma$ corresponding to the Borel
subgroup $(\bB^\sigma)^0$ of $\Gso$.
\end{enumerate}
\end{proposition}
\begin{proof}
We prove (1).
The fact that $\Gso$ is reductive is in \cite[Proposition 1.8]{grnc}.
The description of its root system is in \cite[Proposition 2.1]{qss}. 
Let us describe the coroots.
Let $\bU_\alpha$ be the root subgroup of $\bG$ normalized by $\bT$ on
which $\bT$ acts by $\alpha\in\Sigma$.
For $\{\alpha\in\Sigma\mid \pi(\alpha)\in\Sigma_\sigma\}$
we may choose a set of isomorphisms $u_\alpha:\Ga\xrightarrow\sim \bU_\alpha$
such that $u_{\sigma(\alpha)}=\lexp\sigma(u_\alpha)$.
Then if $\CO(\alpha)$ is not special,  
the root subgroup of $\Gso$ on
which  $\bT^\sigma$ acts  by the  restriction of  $\alpha$ is  the image of
$u_{\pi(\alpha)}:=\prod_i u_{\sigma^i(\alpha)}$.
Then the general formula $\beta^\vee(\lambda)=
u_\beta(\lambda-1)u_{-\beta}(1)u_\beta(\lambda\inv-1)u_{-\beta}(-\lambda)$
applied with $\beta=\pi(\alpha)$ gives $\pi(\alpha)^\vee=\bar\alpha^\vee$
using that $\bU_\alpha$ and $\bU_{\pm\alpha'}$ commute for all
$\alpha\neq\alpha'$ in the same $\sigma$-orbit, by Remark \ref{roots orthogonal}.

If $\CO(\alpha)$ is special (which implies $p\ne 2$),
then $\alpha+\sigma^j(\alpha)$ is a root for
some  $j$.
As the following computation shows, the  root subgroup  of  $\Gso$
on  which  $\bT^\sigma$  acts by
$\pi(\alpha)$  is  isomorphic  to  the  fixed  points of $\sigma^j$ on the
subgroup $\genby{ \bU_\alpha, \bU_{\sigma^j(\alpha)}}$.
The subgroup $\bG_\alpha=
\genby{ \bU_{\pm \alpha},  \bU_{\pm\sigma^j(\alpha)}}$ is isomorphic to
$\SL_3$ or $\PGL_3$. For a certain choice of the morphisms
$u_{\pm(\alpha+\sigma^j(\alpha))}$   the  restriction  of  $\sigma^j$  to
$\bG_\alpha$  is the action of $\text{transpose}\circ\text{inverse}\circ\ad
\begin{pmatrix}0&0&-2\\   0&1&0\\   -2&0&0\end{pmatrix}$   on  $3\times  3$
matrices. With this model, one checks that $(\bG_\alpha^\sigma)^0$ is
isomorphic   to  the  image  of  $\SL_2$  by  the  rational  representation
expressing  the action of $\SL_2$ on homogeneous polynomials of degree two.
If  the  matrix  $\begin{pmatrix}a&b\\c&d\end{pmatrix}$  acts  by $X\mapsto
aX+cY,  Y\mapsto bX+dY$,  on the  basis $\{X^2,XY,Y^2\}$  it acts by
$\begin{pmatrix}a^2&ab&b^2\\2ac&ad+bc&2bd\\c^2&cd&d^2\end{pmatrix}$.  It is
readily  seen on this  model that the  coroot of $\Gso$,  the image of the
coroot of $\SL_2$, is equal to $2\bar\alpha^\vee$.


Let  us  prove  (2).  The  action  of  $W^\sigma$  on  $X_\sigma$ has been
explained  in Lemma \ref{X(Tso)}. The other facts about $(W^\sigma,S_\sigma)$ are
proved in \cite[1.32]{St}.

We prove (3). By \cite[1.8(iv)]{grnc} $N_\bG(\Tso)\subset N_\bG(\bT)$, thus
$N_\Gso(\Tso)$ embeds into $N_\bG(\bT)^\sigma$. Conversely a $\sigma$-fixed
element  which  normalizes  $\bT$  normalizes  $\Ts$  hence  normalizes its
connected  component $\Tso$. Whence  the first assertion  of (3), and the
fact that $W^0(\sigma)$  embeds into  $W^\sigma$. 

That   $(\bB^\sigma)^0$  is  a  Borel   subgroup  of  $\Gso$  is  \cite[1.8
(iii)]{grnc}.  The set $\Sigma_\sigma^+$ is a set of positive roots for the
order induced by $\Sigma^+$. The corresponding root subgroups are subgroups
of $\bB$ as seen in the proof of (1), whence (4).
\end{proof}
\subsection*{Connected components of centralizers}
\begin{proposition}\label{AG(ts)} We have
$\Gs/\Gso\simeq N_\bG(\bT)^\sigma/N_\Gso(\Tso)$ 
\end{proposition}
\begin{proof}

We first show that  $\Gs$ is generated by $\Gso$ and
$\sigma$-stable representatives in $N_\bG(\bT)$ of $W^\sigma$. A
double  coset  $\bB  w\bB$ containing an $x\in\Gs$ is $\sigma$-stable,
hence $w$ is
$\sigma$-stable.  Choose a  representative  $\dot w  $ of $w$ and
write  $x$ uniquely as  $x=ut_1\dot w u_1$  with $u\in\bU\cap\lexp w\bU^-$,
$t_1\in\bT$ and $u_1\in\bU$, where $\bU$ is the unipotent radical of $\bB$.
By uniqueness, $u$ and $u_1$ are in $\bU^\sigma\subset\Gso$ (the inclusion
by \cite[1.8(i)]{grnc}) and $t_1\dot w$ is $\sigma$-stable, so we get that
$\Gs/\Gso$ has representatives in $N_\bG(\bT)^\sigma$.

Since by Proposition \ref{root system of Gtso}(3),
we have $N_\bG(\bT)^\sigma\cap\Gso=N_\Gso(\Tso)$, we get the proposition.
\end{proof}

\begin{notation}
\index{W(sigma)@$W(\sigma)$, $W(t\sigma)$}
Let $W(\sigma):=\{w\in W^\sigma\mid [\sigma,\dot w]\in[\bT,\sigma]\}$ for some
representative $\dot w$ of $w$.
\end{notation}
Note that the condition does not depend on the representative $\dot w$ of
$w$ chosen.

\begin{proposition}\label{exactsequence} The natural morphisms give an exact sequence
$$1\rightarrow \Ts/\Tso\rightarrow
N_\bG(\bT)^\sigma/N_\Gso(\Tso)\rightarrow
W(\sigma)/W^0(\sigma)\rightarrow 1.$$
\end{proposition}
\begin{proof}
Observe  first that $W(\sigma)$ is the set of elements of $W^\sigma$ which
have  representatives in $N_\bG(\bT)^\sigma$ (thus in $N_\Gs(\Tso)$);
indeed $[\sigma,\dot w]\in[\bT,\sigma]$ can be written $\lexp\sigma{\dot w} \dot
w\inv=\lexp\sigma t_1\inv t_1$ for some $t_1\in\bT$ which is equivalent to
$\lexp\sigma(t_1\dot w)=t_1\dot w$.

This observation can be rephrased as the exact sequence
$$1\rightarrow \Ts\rightarrow N_\bG(\bT)^\sigma\rightarrow
W(\sigma)\rightarrow1.$$

By Proposition \ref{AG(ts)} we have $N_\bG(\bT)^\sigma\cap\Gso =N_\Gso(\Tso)$ which is a
normal  subgroup  of  $N_\bG(\bT)^\sigma$.  Taking  the  intersections with
$\Gso$  and  using  that  the  image  in  $W(\sigma)$  of $N_\Gso(\Tso)$ is
$W^0(\sigma)$, we get the exact sequence of the proposition.
\end{proof}

\subsection*{Classification of \qss\ elements}
We want to study $\bG$-orbits (conjugacy classes) of \qss\ 
elements of $\Gun$, and their centralizers in $\bG$. 
We say that two \qss\ elements have {\em same type} if they have
conjugate centralizers, and we want to classify the possible types of \qss\
elements. We can reduce this study to elements of
finite order, thanks to the:

\begin{lemma}\label{finite order}
For any $\sigma\in \Tun$ there exists an element $\sigma'\in \sigma.\Tso$ of
finite order such that $\Gs=\bG^{\sigma'}$.
\end{lemma}
\begin{proof}
Let $\bS$ be the Zariski closure of $\genby\sigma$ in $\genby\sigma
.\Tso$. We have $\bS^0\subset\Tso$,
so $\bS^0$ is diagonalizable, and $\bS/\bS^0$ is generated by the
image of $\sigma$. By the proof of \cite[Lemma 1.6]{cedric}, which uses
only that $\bS^0$ is diagonalizable, the generators of $\bS/\bS^0$ can be lifted to finite
order elements of $\bS$ and for any such $d\in\bS$, there exists $t\in
\bS^0$ such that  $td$ has finite order and satisfies
$C_\bG(td)=C_\bG(\bS)=\Gs$. If we take $d$ lifting $\sigma$ we get
a finite order element $\sigma'=td\in \sigma.\Tso$, which has same centralizer as $\sigma$.
\end{proof}

To classify \qss\ classes, we can fix an element
$\sigma\in\Tun$ of finite order and write the others as $t\sigma$ with $t\in\bT$.
Note that two elements $t\sigma$, $t\sigma'$ of $\Tun$ are $\bG$-conjugate
if and only if $t$ and $t'$ are $\sigma$-conjugate under $\bG$, where
$\sigma$-conjugacy is the ``twisted conjugacy'' $t\mapsto g\inv t\sigma(g)$.
\begin{proposition}\label{omnibus}~
\begin{enumerate}
\item Quasi-semisimple classes of $\Gun$ have representatives in 
$\Tso\cdot\sigma$.
\item  
The $\bG$-conjugacy class of $t\sigma\in\Tso\cdot\sigma$ (or equivalently the
$\sigma$-conjugacy class of $t$) is parameterized by the
$W^\sigma$-orbit of the image of $t$ in
$\LT=\Tso/([\bT,\sigma]\cap\Tso)$; the group $W^\sigma$ acts on 
$\LT$ since it stabilizes $[\bT,\sigma]$.
\item
Given $t\in\Tso$, there exists $t'\in\Tso$, of finite order, such that 
$C_\bG(t'\sigma)=C_\bG(t\sigma)$.
\end{enumerate}
\end{proposition}
\begin{proof}
(1) and (2) are explained in \cite[7.1, 7.2]{grnccomp}. They use that
$\bT=[\bT,\sigma]\cdot\Tso$.

For (3) we use Lemma \ref{finite order} 
which constructs a $t'\sigma\in\sigma\Tso$ of finite
order such that $C_\bG(t'\sigma)=C_\bG(t\sigma)$. Since $\sigma$ commutes
to $t'$ and is of finite order, it is equivalent that $t'\sigma$ be of 
finite order or $t'$ be of finite order.
\end{proof}

By  Proposition \ref{omnibus}(1),    \qss\    classes   have   representatives
$t\sigma\in\Tso\cdot\sigma$  where $\Tso$ is a  maximal torus of $\Gtso$ as
well  as of  $\Gso$. 

In the following, we will describe
representatives  of the  conjugacy classes  of \qss\ elements of
finite  order and  describe their  centralizer, 
which by Proposition \ref{omnibus}(3) suffices to
classify all types of \qss\ elements.

\begin{theorem}\label{parametrization}
\index{Wtilde@$\tW$}
The $\bG$-conjugacy classes of \qss\ elements of finite order in
$\Gun$  are  parameterized  by  the  orbits in $Y_\sigma\otimes\BQ_{p'}$ of 
the  extended affine Weyl group  $\tW:=Y_\sigma\rtimes W^\sigma$, where the
action of $W^\sigma$ is as in Proposition \ref{omnibus}(2).

To describe this parameterization, we fix an isomorphism
$(\BQ/\BZ)_{p'}\xrightarrow\sim \mu(k)$, which lifts to a map
$\iota:\BQ_{p'}\to\mu(k)$,
where $\mu(k)$ is the group of roots of unity in $k$. Then 
the $\tW$-orbit of $\lambda\otimes x$, where $\lambda\in Y_\sigma$ and 
$x\in\BQ_{p'}$, defines $\lambda(\iota(x))\in\LT$ whose preimages
in $\Tso$ are $\sigma$-conjugate.
\end{theorem}
\begin{proof}
By   Proposition \ref{omnibus}(2),   the   \qss\   classes   of  $\Gun$  are
parameterized  by the $W^\sigma $-orbits  in $\LT$. The map $\lambda\otimes
x\mapsto \lambda(\iota(x))$ defines an isomorphism from
$Y(\LT)\otimes(\BQ/\BZ)_{p'}$  to the  group of  points of  finite order of
$\LT$   (see  for instance \cite[0.20]{book}).   Since   $Y(\LT)$  is  equal  to
$Y_\sigma$  by Lemma  \ref{X(T/LT)} we  get that  the \qss\ 
conjugacy  classes  of  finite  order  are  parameterized  by  elements  of
$Y_\sigma\otimes\BQ_{p'}$  up to action of $W^\sigma$ and translation by
$Y_\sigma$. Whence the result.
\end{proof}
Note that since the exponent of $[\bT,\sigma]\cap\Tso$  divides that of
$\sigma$ (see Lemma \ref{A(Ts)}(3)), this group is formed of elements of
finite order, thus identifies with a subset of $Y^\sigma\otimes\BQ/\BZ$.

\begin{proposition}\label{Zso=1}~
\begin{enumerate}
\item  $Z(\bG)^0\subset[\bT,\sigma]$ if and only if $\Zso=\{1\}$.
\item If the conditions of (1) hold the \qss\ $\bG$-orbits
of $\Gun$ and the \qss\ $\bG/Z^0(\bG)$-orbits
of $\Gun/Z^0(\bG)$ are in bijection by the quotient map.
This bijection preserves $W(\sigma)$ and $W^0(\sigma)$.
\end{enumerate}
\end{proposition}
\begin{proof}
If  $\Zso$  is  trivial,  since  $Z(\bG)^0=\Zso.[Z(\bG)^0,\sigma]$ by
Lemma \ref{A(Ts)}(3) applied to the torus $Z(\bG)^0$, we have
$Z(\bG)^0\subset[\bT,\sigma]$.  Conversely if $Z(\bG)^0\subset|\bT,\sigma]$
then $\Zso$ is a connected subgroup of the finite group $\Tso\cap|\bT,\sigma]$,
hence is  trivial. This proves  (1). 

Any  quasi-semisimple class of  $\Gun$ has a  representative in $\Tun$. Two
elements $\sigma,\,t\sigma\in\Tun$ have conjugate images
$\overline\sigma,\,\overline{t\sigma}$  in  $\Gun/Z^0(\bG)$  if  and  only if
there   exists  $g\in\bG$  such  that  $g\sigma g\inv=zt\sigma$  with  $z\in
Z^0(\bG)$. If the conditions of (1) are satisfied we can write $z=t_1\inv\lexp\sigma t_1$
for some $t\in\bT$, so that $t_1g\sigma g\inv t_1\inv=t\sigma$. Hence 
the quotient map induces a bijection on the classes.

Now  for a representative $\dot w$ of $w\in W^\sigma$, the condition
$[\overline\sigma,\overline{\dot  w}]\in [\overline\bT,\overline\sigma]$ is
equivalent  to the  condition $[\sigma,\dot  w]\in[\bT,\sigma]$, whence the
preservation  of  $W(\sigma)$  by  the  quotient  map.  The preservation of
$W^0(\sigma)$  comes from the  fact that $\bG$  and $\bG/Z(\bG)^0$ have the
same root systems and isomorphic root subgroups, thus same
$C_{\sigma,\CO(\alpha)}$ and $\Sigma_\sigma$.
\end{proof}
%

\subsection*{Quasi-central elements}
To classify \qss\ elements it will be useful to relate them to quasi-central
elements, whose definition we recall from \cite[1.15]{grnc}:
\begin{definition}\label{quasicentral}\index{quasi-central@quasi-central}
An element $\sigma$ is quasi-central if $\Gs$ has maximal dimension among
centralizers of elements of $\bG\cdot\sigma$.
\end{definition}
By \cite[1.15(iii)]{grnc} we have:
\begin{proposition}\label{Wsigma=Wosigma}The following are equivalent:
\begin{enumerate}
\item $\sigma$ is quasi-central.
\item $\sigma$ is \qss\ and $W^\sigma=W(\sigma)=W^0(\sigma)$.
\end{enumerate}
\end{proposition}
By (2) above and Proposition \ref{AG(ts)}, Proposition \ref{exactsequence}
reduces when $\sigma$ is quasi-central to $\Gs/\Gso\simeq\Ts/\Tso$.
\begin{remark} When $\sigma$ is quasi-central,
$\Ts=\Tso.Z(\bG^\sigma)$ and
$\Ts/\Tso= Z(\bG^\sigma)/Z(\Gso)$.
\end{remark}
\begin{proof}The second paragraph of the proof of \cite[1.29]{grnc}
gives the first assertion. It remains to show that
$\Tso\cap Z(\bG^\sigma)=Z(\Gso)$.
We have $\Tso\cap Z(\bG^\sigma)\subset Z(\Gso)$. We get the reverse
inclusion using $Z(\Gso)\subset \Tso$ so that
$Z(\Gso)$ commutes with 
$\Gso.\Ts$  which is equal to $\bG^\sigma$ by the 
first paragraph of the proof of \cite[1.29]{grnc}.
\end{proof}

\begin{lemma} There is a quasi-central element of finite order
in $\bT^1$.
\end{lemma}
\begin{proof}
By \cite[1.16]{grnc} there exists a quasi-central element $\sigma_0$
in $\Tun$. By Lemma \ref{finite order} there exists $\sigma\in\Tun$ of finite order
such that $C_\bG(\sigma)=C_\bG(\sigma_0)$, which implies by
definition that $\sigma$ is quasi-central.
\end{proof}
\begin{context}\label{fix qss}
In the rest of the text, we fix a quasi-central element $\sigma\in\Tun$
of finite order and $t$ will denote a finite order element of $\Tso$.
\end{context}
Thus   arbitrary  \qss\   elements  (which   may  be  chosen  in
$\Tso\cdot\sigma$  by Proposition \ref{omnibus}(1))  will now  be denoted  by 
$t\sigma$ --- instead of $\sigma$ as they were before.

We have
\begin{equation}\label{CtsO}
C_{t\sigma,\CO(\alpha)}/C_{\sigma,\CO(\alpha)}=
\prod_{\beta\in\CO(\alpha)}\beta(t)=\bar\alpha(t)
\end{equation}
where $\bar\alpha:=\sum_{\beta\in\CO(\alpha)}\alpha$
and since $t\in\Tso$, we have $\bar\alpha(t)=(\card{\CO(\alpha)}\alpha)(t)$.
\index{alphabar@$\bar\alpha$}

\begin{notation}\label{symplectic type}
\index{orthogonal@orthogonal type}\index{symplectic@symplectic type}
By  \cite [remarks following 1.22]{grnc},  when $\bG$ is quasi-simple there
is  precisely one conjugacy class of quasi-central automorphisms inducing a
given  diagram  automorphism  of  $W$,  except for
groups  of type $A_{2r}$ when $p\neq 2$, in which case  there are two classes
inducing the
non-trivial diagram automorphism of $W$. For one of them $\Gso$ is of type
$B_r$  --- this class disappears if $p=2$ --- 
and  for  the  other of type $C_r$.
In the first case we say that $\sigma$ is {\em of orthogonal type} and
in the second case {\em of symplectic type}
(of course there can be several classes of elements inducing the same class 
of automorphisms, see Corollary \ref{centerR(sigma)}).
\end{notation}

\begin{proposition}
The simple roots of $\Sigma_\sigma$ corresponding to $\Sigma_\sigma^+$ 
are the  $\{\pi(\alpha)\mid  \alpha\in\Pi\}$,  except that  if
$\CO(\alpha)$  is special and  $\pi(\alpha)\notin\Sigma_\sigma$, then one has
to   replace  $\pi(\alpha)$  with  $2\pi(\alpha)$.
\end{proposition}
\begin{proof}
Using \cite[1.32]{St} we get that the basis of $\Sigma_\sigma$ corresponding
to $\Sigma_\sigma^+$  consists 
of those vectors colinear to the vectors of $\pi(\Pi)$ which are in 
$\Sigma_\sigma^+$. This gives the statement.
\end{proof}

\begin{remark}
Note that, if there are no special orbits (in particular if $\Sigma$
has no component of type $A_{2r}$), Definition \ref{Sigmatsigma} and 
Proposition \ref{root system of Gtso} simplify substantially. We have just
$\Sigma_\sigma=\{\pi(\alpha)\mid  \alpha\in \Sigma\}$, and $\Sigma_{t\sigma}$
is  a  subsystem  of  $\Sigma_\sigma$,  defined by the additional condition
$\bar\alpha(t)=1$.

In  type $A_{2r}$, it may happen that $\Sigma_{t\sigma}$ is not a subsystem
of   $\Sigma_\sigma$  for  any   choice  of  $\sigma$.   For  instance,  in
$\bG=\GL_5$,    choose   $\sigma$ as in \cite[page 357]{grnc},
such   that   $\Gso\simeq\Sp_4$  and preserving the diagonal torus.
Number the simple roots
$\alpha_1,\ldots,\alpha_4$ such that
$\alpha_i(s)=\lambda_i/\lambda_{i+1}$ for
$s=\diag(\lambda_1,\ldots,\lambda_5)$.
For $t=\diag(i,1,1,1,-i)\in\Gso$, we get that
$\Sigma_{t\sigma}^+=\{\pi(\alpha_2+\alpha_3),\pi(\alpha_1+\alpha_2)\}$
while $\Sigma_\sigma^+=\{\pi(\alpha_1),\pi(\alpha_2+\alpha_3),
\pi(\alpha_1+\alpha_2+\alpha_3+\alpha_4),
\pi(\alpha_1+\alpha_2+\alpha_3)\}$,   and   for   the   representative
$\sigma'$ of the other  class of quasi-central
automorphisms  whose  fixed  points are $\bG^{\sigma'}=\Orth_5$, we have
$\Sigma_{\sigma'}^+= \{\pi(\alpha_1),\pi(\alpha_2),
\pi(\alpha_1+\alpha_2),\pi(\alpha_1+\alpha_2+\alpha_3)\}$. Thus
$\Sigma_{t\sigma}$ is not a subset of $\Sigma_\sigma$ or $\Sigma_{\sigma'}$.
\end{remark}
\begin{proposition}\label{Gs ss}
We have $Z^0(\Gso)=\Zso$, and
the following are equivalent:
\begin{enumerate}
\item $\Gs$ is semisimple.
\item $Y^\sigma\otimes\BQ=\BQ\Sigma^\vee_\sigma$.
\item $\Zso=1$.
\end{enumerate}
\end{proposition}
\begin{proof}
The first statement follows from \cite[Corollaire 1.25(ii)]{grnc}
which applied for $\bL=\bG$ states that $C_\bG(\Rad\Gso)=\bG$, thus
$Z^0(\Gso)\subset \Zso$. The converse inclusion is obvious.

The equivalence of items (1) to (3) is an immediate consequence of the
first statement and the definitions.
\end{proof}
\begin{remark}
If we do note take the identity component,
it  is not always true that $Z(\Gso)\subset Z(\bG)$. A counterexample is in
$\GL_3$  with  $\sigma(x)=J\lexp  t  x\inv  J\inv$ where $J=\begin{pmatrix}
0&0&-1\\0&-1&0\\1&0&0\end{pmatrix}$.    In   that   case   $\begin{pmatrix}
-1&0&0\\0&1&0\\0&0&-1\end{pmatrix}\in Z(\Gso)$.

Note however that when $\sigma'$ is quasi-central and there is no component
of   type   $A_{2r}$   where   $\sigma'$   is   of  symplectic  type,  then
$Z(\bG^{\sigma'0})\subset   Z(\bG)$.   Indeed,   then   the  
root subgroups  of  $\bG^{\sigma'0}$  involve  all the
root subgroups of $\bG$, thus all roots of
$\bG$ must vanish on a central element in $\bG^{\sigma'0}$.
\end{remark}
We now describe the centralizers of quasi-central elements
for $\bG$ quasi-simple.
\begin{proposition}\label{type of Gso}
Assume that $\bG$ is quasi-simple. Then $\Gso$ is semisimple, and
\begin{enumerate}
\item  If $\bG$ is  adjoint not of type $A_{2r}$ then $\Gs=\Gso$ is adjoint.
\item  If $\bG$ is  simply connected not of type $A_{2r}$ then $\Gs=\Gso$ is
simply connected.
\item When $\bG$ is of type $D_4$ and $n_\sigma =3$, then $\bG$ is adjoint or 
simply connected and $\Gso$ is of type $G_2$.
\item When $\bG$ is of type $D_r$ and $n_\sigma =2$, then $\bG$ is adjoint or 
simply connected or equal to $\SO_{2r}$. In this last case $\Gso$ is adjoint
of type $B_{r-1}$ and $\Gs=\pm 1\cdot\Gso$.
\item When $\bG$ is of type $A_{2r}$ then if $\sigma$ is of symplectic type
then $\Gs=\Gso$ is simply connected of type $C_r$, 
and if $\sigma$ is of orthogonal type then $\Gs=\Gso$ is adjoint of type
$B_r$.
\item When $\bG$ is of type $A_{2r-1}$ and $n_\sigma=2$ then $\Gso$ is of type
$C_r$ and
\begin{itemize}
\item if $d=|P/X|$ is even, where $P$ is the weight lattice of $\Sigma$, then
$\Gso$ is adjoint and $\Gs=\Gso$ except if $d$ divides $r$ and $\car k\ne
2$ in which case $\Gs=\pm 1\cdot\Gso$, where ``$-1$'' is the image in
$\bG$ of a primitive $2d$-th root of unity in the center of $\SL_{2r}$.
\item if $d=|P/X|$ is odd, $\Gs=\Gso$ is simply connected.
\end{itemize}
\end{enumerate}
\end{proposition}
\begin{proof}  
That $\bG$ semisimple implies $\Gso$ semisimple is a consequence of 
Proposition \ref{Gs ss}(3).

Observe  that  since  $\sigma$  is  quasi-central,  proving that
$\Gs=\Gso$  is  equivalent  to  proving  that  $\Ts=\Tso$; this results from
Proposition
\ref{AG(ts)}, the exact sequence of Proposition \ref{exactsequence} and
the fact that $W(\sigma)=W^0(\sigma)$ by Proposition \ref{Wsigma=Wosigma}(2). In
particular,  if  $\bG$  is  adjoint  or  simply connected, then $\Gs=\Gso$
since in these cases we have $\Ts=\Tso$ by Lemma \ref{A(Ts)}(2).

By  Lemma  \ref{X(Tso)}  we  have $X(\Tso)=X_\sigma$. The
weight lattice $P(\Sigma_\sigma)$
\index{P(Sigmas@$P(\Sigma_\sigma)$}
of $\Sigma_\sigma$ is the lattice dual to
$\Sigma_\sigma^\vee$.  If there is no special orbit $\Sigma_\sigma^\vee$
has basis $\{\bar\alpha^\vee\mid\alpha\in\Pi\}$, hence
$P(\Sigma_\sigma)=\pi(P)$ where $P$ is the weight lattice of $\Sigma$. Thus
if $P=X$ then $P(\Sigma_\sigma)=X(\Tso)$ which shows (2).

If there is no special orbit, since $\Sigma_\sigma=\pi(\Sigma)$, we get
that the  root  lattice  $Q(\Sigma_\sigma)$
\index{Q(Sigmas)@$Q(\Sigma_\sigma)$}
of  $\Sigma_\sigma$  is
$\pi(Q)$  where $Q$ is the  root lattice of $\Sigma$. Thus
if $Q=X$ then $Q(\Sigma_\sigma)=\pi(X)=X(\Tso)$ which shows (1).

If  $\bG$  is  of  type  $D_4$,  then  with  the notations of \cite[Planche
IV]{Bou}  the fundamental weights are given in  term of the simple roots by
$\varpi_1=\alpha_1+\alpha_2+\alpha_3/2+\alpha_4/2,
\varpi_2=\alpha_1+2\alpha_2+\alpha_3+\alpha_4,
\varpi_3=\alpha_1/2+\alpha_2+\alpha_3+\alpha_4/2,
\varpi_4=\alpha_1/2+\alpha_2+\alpha_3/2+\alpha_4$.  Thus modulo $Q$ we have
$2\varpi_1=2\varpi_3=2\varpi_4=\varpi_2=0$ and
$\varpi_1=\varpi_3+\varpi_4$.  Hence we  may take  $\varpi_3,\varpi_4$ as a
basis  of $P/Q\simeq(\BZ/2\BZ)^2$. The automorphism $\sigma$ of order
$n_\sigma =3$ of
$\Sigma$  is given by  the permutation $\varpi_1\mapsto  \varpi_3 \mapsto
\varpi_4\mapsto \varpi_1$.  Its matrix in the
basis $\varpi_3,\varpi_4$ of $P/Q$ is
$\begin{pmatrix}0&1\\1&1\end{pmatrix}$, which permutes circularly the three
subgroups of order 2 of $P/Q$. Thus no lattice intermediate between $P$ and
$Q$  is $\sigma$-stable, which  implies that only  the simply connected and
adjoint  group have an algebraic automorphism ``triality'' corresponding to
$\sigma$.

When   $\bG$  is  of  type   $D_{r}$  and  $n_\sigma =2$,  then $\Gso$ is of
type $B_{r-1}$ so that $|P(\Sigma_\sigma)/Q(\Sigma_\sigma)|=|\pi(P)/\pi(Q)|=2$.
If  $r$  is  odd
$P/Q\simeq\BZ/4\BZ$ and there is only one intermediate lattice $Q\subsetneq
X\subsetneq  P$ which  corresponds to  $\SO_{2r}$. If  $r$ is even then
$P/Q\simeq(\BZ/2\BZ)^2$  and there are 3 possible intermediate lattices. In
this  case  the  action  of  $\sigma$  fixes  the  lattice corresponding to
$\SO_{2r}$  and permutes  the two  others (see \cite[page 31]{deriziotis}).
This  gives the first assertion of  (4). If $\bG=\SO_{2r}$, let $\varpi_1$,
\dots,  $\varpi_r$  be  the  fundamental  weights,  with $\sigma$ exchanging
$\varpi_{r-1}$  and $\varpi_r$  (see \cite[Planche  IV, p.\ 256--257]{Bou});
then    $X$    is    spanned    by    $\{\varpi_1,\ldots,\varpi_{r-2},
\varpi_{r-1}-\varpi_r,\varpi_{r-1}+\varpi_r\}$. Thus
$\frac1 2(\varpi_{r-1}+\varpi_r)\in\pi(P)-X_\sigma$ so that
$X_\sigma\neq\pi(P)$.  Since $|\pi(P)/\pi(Q)|=2$  we get $X_\sigma=\pi(Q)$
and $\Gso$ is adjoint. The image $(\sigma-1)X$ is spanned by
$2(\varpi_{r-1}-\varpi_{r})$ while the kernel of $1+\sigma$ is spanned
by $\varpi_{r-1}-\varpi_{r}$, thus according to Lemma \ref{A(Ts)}(1) we have
$|\Ts/\Tso|=2$. A representative is $-1\notin\Tso$, whence (4).

Assume  now  that  $\bG$  is  of  type $A_r$. Then the fundamental group
$P(\Sigma_\sigma)/Q(\Sigma_\sigma)$  has order 2. Hence to show that $\Gso$
is  simply  connected  (resp.~adjoint)  it  is  sufficient  to  show that
$X(\Tso)$ is different from $Q(\Sigma_\sigma)$ (resp.~from
$P(\Sigma_\sigma)$).

Assume first that $\bG$ is of type $A_{2r}$ and that the orbit of $\alpha$ is
special.  If $\sigma$  is of  symplectic type  then $\pi(\alpha)$ is not in
$Q(\Sigma_\sigma)$  but is in $\pi(Q)\subset X(\Tso)$.
Hence $\Gso$ is simply connected.
If  $\sigma$ is  of orthogonal  type, then  $\bar\alpha^\vee$ is not in the
coroot lattice $Q(\Sigma_\sigma^\vee)$
\index{Q(Sigmasv)@$Q(\Sigma_\sigma^\vee)$}
hence $P(\Sigma_\sigma)$ is strictly
larger  than the dual of the lattice spanned by the $\bar\alpha^\vee$ which
contains $X_\sigma=X(\Tso)$, hence $\Gso$ is adjoint.

Assume now that $\bG$  is of type  $A_{2r-1}$. If $|P/X|=d$, then
with  the  notations  of  \cite[Planche  I]{Bou},  $X$ is generated by
$\alpha_1,\ldots,\alpha_{2r-1},d\varpi_{2r-1}$,   thus   $X_\sigma$   is
generated by $\pi(Q)$ and $\pi(d\varpi_{2r-1})$. Thus $X_\sigma=\pi(Q)$ if and
only if $\pi(d\varpi_{2r-1})=\frac
d2(\alpha_1+\ldots+\alpha_{2r-1})\in\pi(Q)$, whence the result since $\frac
d2(\alpha_1+\ldots+\alpha_{2r-1})\equiv \frac d2\alpha_r\pmod{\pi(Q)}$.

We prove finally the assertions on $\Gs$ in cases (5) and (6).
Let $\bG$ be of type $A_l$ and  let $d=|P/X|$, a divisor of $l+1$.
We keep the notation from \cite[Planche  I]{Bou}:
the lattice $X$ is spanned by $Q$ and $d\varpi_l$.
We  want to  use Lemma  \ref{A(Ts)}(1). For  computing $\Ker(\sigma+1)$ and
$(\sigma-1)X$ we introduce a suitable basis of $X$. Let
$i_1:=2(\alpha_{l-1}-\alpha_2)+3(\alpha_{l-2}-\alpha_3)+\cdots+
\lfloor\frac{l}{2}\rfloor(\alpha_{\lfloor
\frac{l+3}{2}\rfloor}-\alpha_{\lfloor\frac{l}{2}\rfloor})$
and $i_0:=(\alpha_l-\alpha_1)+i_1$.  We use the notation $\lambda:=\begin{cases}
0&\text{if $l$ is even}\\\frac{d}{2}\alpha_{(l+1)/2}&\text
{if $l$ is odd}\end{cases}$  and  $\mu:=-\frac{d}{l+1}  i_0+\lambda$;  then  we
claim  that $\{\alpha_2,\alpha_3,\ldots,\alpha_l,\mu\}$ is a basis of $X$.
Indeed  $X$ has rank $l$, spanned by $\{\alpha_1,\ldots,\alpha_l,d\varpi_l\}$
and  we  have $\alpha_1=\frac{l+1}{d}(\mu-\lambda)+i_1+\alpha_l$ and
$d\varpi_l=\frac{d}{l+1}(\alpha_1+2\alpha_2+3\alpha_3+\cdots+l\alpha_l)=
\mu+d(\alpha_{\lfloor\frac{l+3}2\rfloor}+\ldots+\alpha_l)$  so
that $\mu\equiv d\varpi_l\pmod Q$. We have $\sigma(i_0)=-i_0$ and
$\sigma(\lambda)=\lambda$ so that $\sigma(\mu)=\mu+\frac{2d}{l+1}i_0$.
We deduce that $(\sigma-1)X$ is spanned by
$\{\alpha_l-\alpha_1,\alpha_{l-1}-\alpha_2,\ldots,\alpha_{\lfloor
\frac{l+3}{2}\rfloor}-\alpha_{\lfloor\frac{l}{2}\rfloor},
\frac{2d}{l+1}i_0\}$ that is by
$\{\alpha_{l-1}-\alpha_2,\ldots,i_0,\frac{2d}{l+1}i_0\}$, since 
$\alpha_l-\alpha_1=i_0-i_1$ and $i_1$ is an integral combination
of $\{\alpha_{l-1}-\alpha_2,\ldots\}$.

We compute now the kernel of $1+\sigma$. We have
$(1+\sigma)(a_2\alpha_2+\cdots+a_l\alpha_l+b\mu)=
(a_2+a_{l-1})(\alpha_2+\alpha_{l-1})+(a_3+a_{l-2})(\alpha_3+\alpha_{l-2})+\cdots
+a_l(\alpha_l+\alpha_1)+2b\lambda\}$.
Since in this sum $\alpha_l+\alpha_1$ is the only term with a non-zero coefficient 
on $\mu$ when written in the above basis of $X$,
the elements of the kernel of $1+\sigma$ have $a_l=0$, $a_i=-a_{l+1-i}$  for
$i=2,\ldots,\lfloor\frac l2\rfloor$ and, if $l$ is odd $db=2a_{\frac{l+1}2}$ (otherwise $b$ is
arbitrary). We get that $\ker(\sigma+1)$ is spanned by
the $\alpha_i-\alpha_{l+1-i}$ for $i=2,\ldots,\lfloor\frac l2\rfloor$ and $\mu-\lambda$ if
$\mu-\lambda\in X$ otherwise by the $\alpha_i-\alpha_{l+1-i}$ and $2(\mu-\lambda)$.
Since $\mu-\lambda$ is in $X$ if and only if $\lambda=0$ or $l$ is odd, we get 
$\ker(1+\sigma)=\genby{(\alpha_i-\alpha_{l+1-i})_{i=2,\ldots,\lfloor\frac l2\rfloor},
\frac d{l+1}i_0}$ if $l$ is even or $d$ is
even, and
$\ker(1+\sigma)=\genby{(\alpha_i-\alpha_{l+1-i})_{i=2,\ldots,\lfloor\frac
l2\rfloor},\frac {2d}{l+1}i_0}$
if $l$ and $d$ are odd. Hence $\ker(1+\sigma)/(\sigma-1)X$ is non-trivial 
(of order $2$) only
if $l$ is even or $d$ is even and $\frac d{l+1} i_0$ is not in the span of
$\genby{i_0,\frac{2d}{l+1}i_0}$ thus if and only if $d$ is even and $2d$ divides $l+1$.
Writing $l=2r-1$ in that case, we get $2d=2|P/X|$ divides $2r$ that is $|P/X|$ divides $r$.
\end{proof}
\section{The root datum $R(\sigma)$ and its affine Weyl group}
We recall that by Context \ref{fix qss}, $\sigma$ denotes a quasi-central
element (a more general \qss\ element is denoted by $t\sigma$).
\begin{context}\label{sigma symplectic}
From now on, when there is a component $\bH$ of $\bG$ of type $A_{2r}$ and
$p\ne 2$, 
we assume that the smallest power $\sigma^i$ which stabilizes that component
is such that $\bH^{\sigma^i}$ is of type $A_{2r}$ or $C_r$
(and {\em not} of type $B_r$), that is on that component $\sigma^i$ acts
trivially or is of symplectic type. 
\end{context}
Note that when $p=2$ there is no choice to make:
$\bH^{\sigma^i}$ is of type $A_{2r}$ if $\sigma^i$ acts trivially on $\bH$
and of type $C_r$ otherwise.
Context \ref{sigma symplectic}
chooses in every case a well-defined conjugacy class of
quasi-central automorphisms $\sigma$.
\begin{proposition}\label{Csigma,O(alpha)} It follows from the choice
made in Context \ref{sigma symplectic} that for $\alpha\in\Sigma$ we
have $C_{\sigma,\CO(\alpha)}=1$, unless $\CO(\alpha)$ is special and $p\neq 2$,
in which case $C_{\sigma,\CO(\alpha)}=-1$.
\end{proposition}
\begin{proof}
This follows from \cite[1.8(v) and second example page 357]{grnc} and from
\cite[2.5]{qss}.
\end{proof}
We   now  introduce  a  root datum 
$R(\sigma)=(\Phi_\sigma,\Phi_\sigma^\vee, X^\sigma,Y_\sigma)$, 
(recall from Lemma \ref{X(T/LT)} that $X^\sigma=X(\LT)$ and $Y_\sigma=Y(\LT)$) 
which  will have  the advantage over $\Sigma_\sigma$
that the system $\Phi_{t\sigma}$ associated to
$t\sigma\in\Tso\cdot\sigma$ will be a
closed   subsystem  of   $\Phi_\sigma$, closely related to $\Sigma_{t\sigma}$.  The   root  system  introduced  in
\cite[Proposition  2]{Sp} is equal to $\Phi_\sigma$ if there are no special
orbits, but is not as well behaved when they exist.
\begin{proposition}\label{Phi_sigma}\hfill\break
$$\Phi_\sigma:=\{\bar\alpha\mid\alpha\in\Sigma,\CO(\alpha)\text{ not special
or cospecial}\}
\cup\{2\bar\alpha\mid\alpha\in\Sigma,\CO(\alpha)\text{ special}\}
\subset X^\sigma$$
and 
\index{Phis@$\Phi_\sigma$, $\Phi^\vee_\sigma$}
$\Phi^\vee_\sigma:=\{\pi(\alpha^\vee)\mid\alpha\in\Sigma,\CO(\alpha)
\text{ is not cospecial}\}\subset  Y_\sigma$, 
are the roots and the coroots of a reduced root
system for $W^\sigma$ acting as in Proposition \ref{omnibus}(2); the root
corresponding to the coroot $\pi(\alpha^\vee)$ is $\bar\alpha$ or
$2\bar\alpha$.

The set $\Phi_\sigma^+:=\Phi_\sigma\cap\BQ^+\Pi$ is a set of
positive roots for $\Phi_\sigma$ and
$$\Delta_\sigma:=\{\bar\alpha\mid\alpha\in\Pi,\CO(\alpha)\text{ not special}\}
\cup\{2\bar\alpha\mid\alpha\in\Pi,\CO(\alpha)\text{ special}\}$$
is the corresponding set of simple roots.
\index{Ds@$\Delta_\sigma$}
The weight lattice $P(\Phi_\sigma)$ is equal to $P^\sigma$ 
\index{p(Phis)@$P(\Phi_\sigma)$}
\index{P@$P$}
where $P$ is the weight lattice of $\Sigma$.
The fundamental coweights of the root system  $(\Phi_\sigma,\Phi_\sigma^\vee)$
are $\{\pi(\varpi_\alpha^\vee)\mid \alpha\in\Pi,\CO(\alpha)\text{ not
special}\}\cup\{\frac12\pi(\varpi_\alpha^\vee)\mid\alpha\in\Pi,
\CO(\alpha)\text{ special}\}$, where $\{\varpi_\alpha^\vee\mid\alpha\in\Pi\}$
are the fundamental coweights of the root system $\Sigma$.
\end{proposition}
\begin{proof}
By  \cite[2.3]{qss} and its proof
the  set $\Phi'_\sigma$  obtained from $\Phi_\sigma$ by
replacing  $2\bar\alpha$ with $\bar\alpha$  for special orbits,  is a reduced
root system in $X^\sigma\otimes\BQ$ for the group $W^\sigma$.
Since $W^\sigma$  maps a special  orbit to a special
orbit,  it  stabilizes  $\Phi_\sigma$. Since the vectors in $\Phi_\sigma$ are 
colinear to the vectors in $\Phi'_\sigma$, the reflections of $W^\sigma$ are
reflections with respect to the elements of $\Phi_\sigma$.
To  prove  that $\Phi_\sigma$ is a root
system  and  $\Phi_\sigma^\vee$  is  the  corresponding  set  of coroots it
suffices to prove that the
pairing between $X^\sigma\otimes\BQ$ and $Y_\sigma\otimes\BQ$ takes
integral values on $\Phi_\sigma\times\Phi_\sigma^\vee$ with diagonal values
equal to 2.

For  all  roots  $\alpha$  and  $\beta$ we have $\pairing{\pi(\beta^\vee)}
{\bar\alpha}
=\frac1{|\CO(\beta)|}\sum_{\gamma\in \CO(\beta)}\pairing{\gamma^\vee}
{\bar\alpha}$.
Since the  pairing is $\sigma$-invariant, this is equal
to  $\pairing{\beta^\vee}{\bar\alpha}$ which is an integer. In particular
$\pairing{\pi(\alpha^\vee)}{\bar\alpha}=\pairing{\alpha^\vee}{\bar\alpha}$;
by Remark \ref{roots orthogonal} this is equal to
2 if $\alpha$ is not special and to 1 if $\alpha$ is special, so that in the
latter case $\pairing{\pi(\alpha^\vee)}{2\bar\alpha}=2$.
We have proved the assertions
on $\Phi_\sigma$ and $\Phi_\sigma^\vee$.

The assertions on the positive and simple roots result from the fact that the
root $\bar\alpha$ is a positive multiple of $\pi(\alpha)$.

To  prove the assertion  on the weights,  first notice that the fundamental
weights  $\{\varpi_\alpha\mid \alpha\in\Pi\}$  are permuted  by $\sigma$
since they are the dual basis to the simple coroots. Thus
$\{\sum_{\beta\in\CO}\varpi_\beta\}_{\CO\in\Pi/\sigma}$   is  a   basis  of
$P^\sigma$. Now
for all simple roots $\alpha$ and $\gamma$
$$\sum_{\beta\in\CO(\alpha)}\varpi_\beta(\pi(\gamma^\vee))=     \begin{cases}
1&\mbox{  if $\gamma\in\CO(\alpha)$ and  $\CO(\alpha)$ is not cospecial,}\\
0&  \mbox{ otherwise.} \end{cases} $$ This shows that $P^\sigma$ is
indeed dual to $\BZ\Phi_\sigma^\vee$.

The assertion on coweights comes from a similar computation: 
for $\alpha,\beta\in\Pi$ we have
$\pairing{\pi(\varpi_\alpha^\vee)}{\bar\beta}=\pairing{\varpi_\alpha^\vee}
{\bar\beta}=
\delta_{\CO(\alpha),\CO(\beta)}$, by invariance of the pairing. If $\CO(\alpha)$ is not
special then $\bar\alpha\in\Delta_\sigma$ and the corresponding fundamental coweight
is $\pi(\varpi^\vee_\alpha)$; if $\CO(\alpha)$ is special then $2\bar\alpha\in\Delta_\sigma$
and the corresponding fundamental coweight is $\frac12\pi(\varpi^\vee_\alpha)$.
\end{proof}
The following proposition is, for the root datum 
$R(\sigma)$, the analogue  of Proposition \ref{type of Gso} for
the root datum $(\Sigma_\sigma,\Sigma^\vee_\sigma,X_\sigma,Y^\sigma)$
\begin{proposition}\label{R(sigma)}
Let $R(\sigma)$ be the root datum
\index{Rs@$R(\sigma)$}
$(\Phi_\sigma,\Phi^\vee_\sigma,X^\sigma,Y_\sigma)$. 
Assume that $\bG$ is quasi-simple; then
$Y_\sigma\otimes\BQ=\BQ\Phi^\vee_\sigma$,  that is $R(\sigma)$ is
semisimple, and
\begin{enumerate}
\item  If $\bG$ is  adjoint not of type $A_{2r}$ then $R(\sigma)$ is of adjoint type.
\item  If $\bG$ is  simply connected then $R(\sigma)$ is
of simply connected type.
\item If $\bG=\SO_{2r}$ and $n_\sigma=2$, then $R(\sigma)$ is of type $C_{r-1}$
simply connected.
\item If $\bG$ is of type $A_{2r}$ and $n_\sigma=2$
then $R(\sigma)$ is of type  $C_r$ simply connected.
\item When $\bG$ is of type $A_{2r-1}$ and $n_\sigma=2$ then if $|P^\vee/Y|$ is even,
where $P^\vee$ is the coweight lattice of $\Sigma$, then
$R(\sigma)$ is of type $B_r$ simply connected; otherwise $R(\sigma)$ is of type $B_r$ adjoint.
\end{enumerate}
\end{proposition}
\begin{proof}
That $\bG$ semisimple implies $R(\sigma)$ semisimple
is an immediate consequence of Proposition \ref{Gs ss}(2) since
$Y_\sigma\otimes\BQ=\BQ\Phi^\vee_\sigma$ is clearly equivalent to
$Y^\sigma\otimes\BQ=\BQ\Sigma_\sigma^\vee$.

Recall that $P$ and $Q$ denote the weight and root lattice of
$\Sigma$. We denote by $P^\vee$ and  $Q^\vee$ \index{P @$P^\vee$}\index{Q @$Q^\vee$}
the coweight and  coroot lattice of $\Sigma$.

If there is no special orbit the root lattice $Q(\Phi_\sigma)$ of
$\Phi_\sigma$ is equal to $Q^\sigma$.
\index{Q(Phis)@ $Q(\Phi_\sigma)$}
Hence, if $X=Q$ then $X^\sigma=Q(\Phi_\sigma)$, whence (1).

Similarly, if $X=P$ then $X^\sigma=P^\sigma=P(\Phi_\sigma)$,
the last equality by Proposition \ref{Phi_sigma}, whence (2) for type
not $A_{2r}$. For this last type, see the proof of (4).

If $\bG=\SO_{2r}$ we use the same computation as in the proof of
Proposition \ref{type of Gso}(4), replacing weights with coweights
and $X$ with $Y$. This gives
$Y_\sigma\neq\pi(P^\vee)$ so that
$Y_\sigma=\pi(Q^\vee)$ since $|\pi(P^\vee)/\pi(Q^\vee)|=2$,
whence (3) since $\pi(Q^\vee)$ is the coroot lattice of
$\Phi_\sigma$. There is one orbit of simple
roots with 2 elements, which gives a long simple root in $\Phi_\sigma$
so that the type of $R(\sigma)$ is $C_{r-1}$.

If $\bG$ is of type $A_r$ we have $|P(\Phi_\sigma^\vee)/Q(\Phi_\sigma^\vee)|=2$.
Hence to show that $R(\sigma)$ is of adjoint type (resp.~of simply connected type) it is
sufficient to show that $Y_\sigma$ is different from $Q(\Phi_\sigma^\vee)=
\pi(Q^\vee)$ (resp.~from $P(\Phi_\sigma^\vee)=\pi(P^\vee)$).

If $\bG$ is of type $A_{2r}$ and $n_\sigma=2$, then
for $\alpha\in\Pi$ such that $\CO(\alpha)$ is special,  
$\frac12\pi(\varpi_\alpha^\vee)=\frac14(\varpi_\alpha+\varpi_{\sigma(\alpha)})$
is a coweight of $\Phi_\sigma$ but is not in $\pi(P^\vee)$,
hence is not in $Y_\sigma=\pi(Y)$. Thus $Y_\sigma\neq \pi(P^\vee)$, and
$R(\sigma)$ is of simply connected type. There is one special orbit of simple roots
which gives a long simple root in $\Phi_\sigma$ so that $R(\sigma)$ is of type
$C_r$.

Assume $\bG$ of type $A_{2r-1}$ and $n_\sigma=2$.
Since exchanging the roots and the coroots in type $A$ gives an isomorphic
root system, we can make the same computation as in the proof of Proposition
\ref{type of Gso}(6), exchanging $Q$ with $Q^\vee$,  $P$ with $P^\vee$,
and $X$ with $Y$. This gives that $Y_\sigma=\pi(Q^\vee)$ if and only if 
$|P^\vee/Y|$ is even, whence the result. There is no special orbit of simple roots 
and one orbit is a singleton, so that this orbit gives a short root in $\Phi_\sigma$
and the type is $B_r$.
\end{proof}

We recall that
$(W^\sigma,S_\sigma)$ is a Coxeter system (see Proposition \ref{root system of Gtso}(2)).

\begin{proposition}\label{aff}\index{I@I}
Let 
$\Wa:=Q(\Phi_\sigma^\vee)\rtimes W^\sigma$ be the affine Weyl group where
\index{Wa@$\Wa$}
\index{affine Weyl group}
\index{Q(phisv)@$Q(\Phi_\sigma^\vee)$}
$Q(\Phi_\sigma^\vee)=\BZ\Phi_\sigma^\vee$ is the coroot lattice.
Let  us denote by  $(W^\sigma_i)_{i\in I}$ the  irreducible components of
\index{Wsi@$W^\sigma_i$}
$W^\sigma$  and  by  $\Phi_{\sigma,i}$  (resp.~$S_{\sigma,i}$)  the 
corresponding subsets of
\index{Phisi@$\Phi_{\sigma,i}$}
\index{Ssi@$S_{\sigma,i}$}
$\Phi_\sigma$ (resp.~$S_\sigma$).
For $i\in I$, let $s_{0,i}$ be the reflection with respect to the 
\index{S0i@$s_{0,i}$}
hyperplane $\alpha_{0,i}(x)=1$, 
\index{a0i@$\alpha_{0,i}$}
where $\alpha_{0,i}$ is the highest root of $\Phi_{\sigma,i}$ for
the order $\Phi_\sigma^+$. Finally,  let $\tilde S_{\sigma,i}:=S_{\sigma,i}
\index{Stsi@$\tilde S_{\sigma,i}$}
\cup\{s_{0,i}\}$ and let $\tilde S_\sigma=
\cup_i\tilde S_{\sigma,i}$.
\index{Sts@$\tilde S_\sigma$}
\begin{enumerate}
\item
$(\Wa,\tilde S_\sigma)$ is a Coxeter system.
\item 
The closure $\CC$ of the fundamental alcove 
\index{C@$\CC$}
of the hyperplane system of $\Wa$ is a
fundamental domain for $\Wa$ acting on $\BQ\Phi_\sigma^\vee$. It is the product of the 
simplices $\CC_i$ with vertices 
\index{Ci@$\CC_i$}
$0$ and $\{\frac1{n_s}\varpi_s^\vee\mid s\in S_{\sigma,i}\}$, where
$\{\varpi_s^\vee\}_{s\in S_\sigma}$ are the fundamental coweights of the 
\index{osv@$\varpi_s^\vee$}
root system $(\Phi_\sigma,\Phi_\sigma^\vee)$  and 
$\sum_{s\in S_{\sigma,i}} n_s\alpha_s$ \index{ns@$n_s$} is the
expression of $\alpha_{0,i}$ in term of the simple roots.
\item 
$\tilde S_\sigma$ is the set of
reflections with respect to the walls of $\CC$.
\item
\index{Dst@$\tilde\Delta_\sigma$}
Let $\tilde\Delta_\sigma:=\Delta_\sigma\cup\{-\alpha_{0,i}\mid i\in I\}$.
For each $i\in I$, let $J_i:=\{s\in S_{\sigma,i}\mid  n_s=1\}$.
\index{Ji@$J_i$}
and for $s\in J_i$ let $z_s:= w_{S_{\sigma,i}-\{s\}}w_{S_{\sigma,i}}$; then $z_s$
\index{zs@$z_s$}
is an automorphism of $\tilde\Delta_\sigma$ and 
$\prod_{i\in I}(\{z_s\mid s\in J_i\}\cup\{\Id\})$ is the group
of all automorphisms of $\tilde\Delta_\sigma$ induced by $W^\sigma$.
\item 
Let  $\Wa':=P(\Phi_\sigma^\vee)\rtimes W^\sigma$ be  the semi-direct product
\index{Wap@$\Wa'$}
\index{P(phisv)@$P(\Phi_\sigma^\vee)$}
by the coweight lattice, giving a natural isomorphism 
$j:\Wa'/\Wa\xrightarrow\sim P(\Phi_\sigma^\vee)/Q(\Phi_\sigma^\vee)$.
Then $\prod_{i\in I}(\{0\}\cup\{\varpi_s^\vee\mid s\in J_i\})$ is a
set of representatives of
$P(\Phi_\sigma^\vee)/Q(\Phi_\sigma^\vee)$.
For $s\in J_i$, let  $\gamma_s$ be the composed of $z_s$ and the
\index{gs@$\gamma_s$}
translation  by $\varpi^\vee_s$;
then $\gamma_s(\CC_i)=\CC_i$.
Composing $j$ with the Cartesian product over $I$ of the maps
which send $\varpi^\vee_s$ to $\gamma_s$ and $0$ to $\Id$ we 
get a group isomorphism from $\Wa'/\Wa$ to the automorphisms of $\CC$ induced 
by $\Wa'$.
\end{enumerate}
\end{proposition}
A consequence of (5) is that if $\Gs$ is semisimple,
$\tW$ is generated by $\Wa$ and some of the $\gamma_s$.

Note also that $\gamma_s\mapsto z_s$ is a bijection, thus (4) and (5) establish
isomorphisms between the groups $\Wa'/\Wa$, $\Aut_{W^\sigma}(\tilde\Delta_\sigma)$,
$\Aut_{\Wa'}(\CC)$ and $P(\Phi_\sigma^\vee)/Q(\Phi_\sigma^\vee)$.
\begin{proof}Statement (1) comes from \cite[Ch.\ V, \S3, Th.1]{Bou} and
\cite[Ch.\ VI \S2, Prop.5]{Bou}.

The fact that $\CC$ is a fundamental
domain, and its decomposition as a product of simplices is stated in
\cite[Ch.\ VI, \S2, end of no 1]{Bou}. The formula for the vertices of the
fundamental alcove of
each irreducible component is given in \cite[Ch.\ VI, \S2 Cor.\ of Prop.5]{Bou}

Statement (3) is \cite[Ch.\ VI, \S2 Prop.5 (i)]{Bou}.

Statement (4) is \cite[(3.5)]{cedric}

Statement (5) is deduced from
\cite[Ch.\ VI, \S2 no 3, Prop.6 and its corollary]{Bou} which discusses the
case of an irreducible Coxeter group. 

One gets the consequence of (5) mentioned after the statement using that, 
as observed at the beginning of the proof of Proposition \ref{R(sigma)},
if $\Gs$ is semisimple then $R(\sigma)$ also, thus $\tW$ is a subgroup of $\Wa'$.
\end{proof}
Note that Proposition \ref{aff}(5) allows when $\Gs$ is semisimple to see the 
fundamental domain of $\tW$ as a subset of $\CC$.

With the convention of Context \ref{sigma symplectic}, we have:
\begin{corollary}\label{root system of Gtso2}
The $\bG$-conjugacy classes of \qss\ elements of finite order in
$\Gun$  are parameterized by the  points of a fundamental domain of
$\tW$ in $Y_\sigma\otimes\BQ_{p'}$. If the class of $t\sigma$ is parameterized
by $\lambda$, the root system $\Sigma_{t\sigma}$
of $\Gtso$ is 
$$\left\{\pi(\alpha), \alpha\in\Sigma\text{ such that $\CO(\alpha)$ is }
\left|
\begin{array}{l}
\text{ special and }2\bar\alpha(\lambda)\in 2\BZ+1\\
\text{ not special and }\bar\alpha(\lambda)\in\BZ\\
\end{array}\right.\right\}.$$
\end{corollary}
Note that the conditions above
corresponds to $\lambda$ lying on one of the affine hyperplanes of $\Wa$.
\begin{proof}
The first assertion is clear from Theorem \ref{parametrization}.
We get the second assertion by noticing that the
condition    which   defines   
$\Sigma_{t\sigma}$ in Definition \ref{Sigmatsigma}
can   be   written $\bar\alpha(\lambda)\in\BZ$  when
$C_{\sigma,\CO(\alpha)}=1$.
Similarly, when $C_{\sigma,\CO(\alpha)}=-1$ the condition can be written
$2\bar\alpha(\lambda)\in 2\BZ+1$.
\end{proof}

\begin{proposition}\label{Phi_tsigma}
$\Phi_{t\sigma}:=\{\alpha\in\Phi_\sigma\mid \alpha(t)=1\}$ is a closed subsystem
of $\Phi_\sigma$, and is a root system for the
\index{Phits@$\Phi_{t\sigma}$}
reflection subgroup $W^0(t\sigma)$ of $W^\sigma$. In other terms, if
$\lambda\in Y_\sigma\otimes\BQ$ parameterizes the class of $t\sigma$ as in 
Corollary \ref{root system of Gtso2},
we have $W^0(t\sigma)=\genby{s_\alpha\in W^\sigma, \alpha\in\Phi_\sigma\mid
s_\alpha(\lambda)-\lambda\in\BZ\Phi_\sigma^\vee}$.
\end{proposition}
In the above $\alpha(t)$ is the evaluation of $\alpha\in X^\sigma$
on $t$.
\begin{proof}
$\Phi_{t\sigma}$ is a closed subsystem of
$\Phi_\sigma$ since the condition which defines it is clearly closed. 
Let us show that the roots of $\Phi_{t\sigma}$ and $\Sigma_{t\sigma}$ are
colinear. If $\CO(\alpha)$ is not special or cospecial then $\bar\alpha$ is a
root of $\Phi_\sigma$ and by Equation \ref{CtsO} the condition $\bar\alpha(t)=1$
is equivalent to $\pi(\alpha)\in\Sigma_{t\sigma}$, thus
$\bar\alpha\in\Phi_{t\sigma}$ if and only if $\pi(\alpha)\in\Sigma_{t\sigma}$.
If $\CO(\alpha)$ is special
the corresponding root of $\Phi_\sigma$ is $2\bar\alpha$ so the condition
for $2\bar\alpha$ to be in $\Phi_{t\sigma}$ is $\bar\alpha(t)=\pm 1$.
Each of the signs determines a root in $\Sigma_{t\sigma}$
colinear to $\bar\alpha$:  with our choice of 
$\sigma$ (see Context \ref{sigma symplectic}) in case $A_{2r}$ this root is 
$\pi(\alpha)$ if 
$\bar\alpha(t)=-1$, and $\pi(\alpha+\sigma^i(\alpha))=2\pi(\alpha)$ otherwise, 
where $2i=|\CO(\alpha)|$.
Thus by \cite[remark after 8.5]{grnccomp}, the root systems $\Phi_{t\sigma}$
and $\Sigma_{t\sigma}$
have same Weyl group, which is $W^0(t\sigma)$ by Proposition
\ref{root system of Gtso}.
Finally $s_\alpha(\lambda)-\lambda=-\alpha(\lambda)\alpha^\vee$ thus
$\alpha(\lambda)\in\BZ$ is equivalent to 
$s_\alpha(\lambda)-\lambda\in\BZ\Phi_\sigma^\vee$.
\end{proof}
\begin{corollary}\label{centerR(sigma)}
The $\bG$-classes of quasi-central elements in $\Gun$ are in bijection with 
the center of $R(\sigma)$, where
we define the center of the root datum $R(\sigma)$ as the
set of orbits of $\tW/\Wa$ on 
$\prod_{i\in I}(\{0\}\cup\{\varpi_s^\vee\mid s\in J_i\})$.
\end{corollary}
We note in particular that though for a quasi-simple group not of type
$A_{2n}$ there is a single
class of quasi-central automorphisms by \cite[1.22]{grnc}, there is
more than one class of quasi-central elements whenever the datum 
$R(\sigma)$ is not adjoint,
which,  as described in Proposition \ref{R(sigma)}, often happens.
\begin{proof}
The   class   of   $t\sigma\in\Gun$   is   quasi-central  if  and  only  if
$W^0(t\sigma)=W^\sigma$,  hence,  by  Proposition  \ref{Phi_tsigma}, if and
only    if   $s_\alpha(\lambda)-\lambda\in\BZ\Phi_\sigma^\vee$    for   all
$\alpha\in\Phi_\sigma$,  where  $\lambda\in\CC$  parameterizes  the class of
$t\sigma$.   This   is   equivalent   to  $\alpha(\lambda)\in\BZ$, which is
equivalent to the same condition for  all $\alpha\in\Delta_\sigma$.
Since  $\lambda$ is in $\CC$ we
have $\lambda=\sum_{i\in I}\sum_{s\in S_{\sigma,i}}
\frac{\lambda_s}{n_s}\varpi_s^\vee$  with $\lambda_s\geq 0$ and $\sum_{s\in
S_{\sigma,i}}\lambda_s=1$  for  any  $i$.  If  $\alpha\in\Delta_\sigma$
corresponds to $s\in S_\sigma$ one has
$\alpha(\lambda)=\frac{\lambda_s}{n_s}\in[0,\frac1{n_s}]$,     hence
$\alpha(\lambda)\in\BZ$   if  and  only  if $\lambda_s=0$ or
$\lambda_s=n_s=1$. From the condition $\sum_{s\in S_{\sigma,i}}\lambda_s=1$
we  deduce that for each $i$ there is at most one non-zero $\lambda_s$ with
$s\in    S_{\sigma,i}$,   which   means    that   $\lambda\in   \prod_{i\in
I}(\{0\}\cup\{\varpi_s^\vee\mid s\in J_i\})$.
By Corollary \ref{root system of Gtso2}, the quasi-central classes
are parameterized by the $\tW$-orbits of such $\lambda$.
\end{proof}
We call {\em minuscule}  the weights of the form
$\sum_{i\in K}\varpi_{s_i}$  where $K\subset I$ is non-empty
and where, for each $i\in K$ the fundamental weight 
$\varpi_{s_i}$ corresponds to $s_i\in J_i$ (this
extends the usual definition of minuscule weights in the irreducible case).
\index{minuscule}
\begin{lemma}\label{minuscule}
Let $Q$ be the root lattice of $\Sigma$ in $X\otimes\BQ$.
\index{Q@$Q$}
Then the non-zero elements of $P^\sigma/Q^\sigma$ identify with certain
$\sigma$-stable minuscule weights.
If $\bG$ is semisimple we have $Q\subset X\subset P$, thus
if $\sigma$ fixes no minuscule weight, then $Q^\sigma=X^\sigma=P^\sigma$.
\end{lemma}
\begin{proof}
When  $\Sigma$ is irreducible,  a $\sigma$-stable set  of representatives of 
the
non-zero elements of $P/Q$ is given by the minuscule weights (see \cite[VI,
\S2,  exercice 5]{Bou}); thus  in general such  a set is  given by the 
minuscule weights with our extended definition.
The  first
assertion  is  thus  a  consequence  of  the  exact  sequence $0\rightarrow
Q^\sigma\rightarrow  P^\sigma\rightarrow(P/Q)^\sigma$. The second assertion
follows.
\end{proof}

\begin{example}
Consider the case of $\bG$ semisimple of type $D_4$ with the triality
automorphism $\sigma$. 
There are two subcases, $\bG$ simply connected or adjoint ---
the intermediate cases do not admit a triality, see Proposition \ref{type of Gso}(3).
\begin{itemize}
\item When $\bG$ is adjoint, $X=Q$ has basis the simple roots
$\alpha_1,\alpha_2,\alpha_3,\alpha_4$
and $\sigma$ is given by the cycle
$\alpha_1\mapsto\alpha_3\mapsto\alpha_4$.
\item When $\bG$ is simply connected, $X=P$ has basis the weights
$\varpi_1,\varpi_2,\varpi_3,\varpi_4$
and $\sigma$ is given by the cycle
$\varpi_1\mapsto\varpi_3\mapsto\varpi_4$.
\end{itemize}
The minuscule weights are $\varpi_1,\varpi_3,\varpi_4$ and none is fixed by
$\sigma$. By the second assertion of Lemma \ref{minuscule}, it follows that
if we set $\alpha:=\alpha_1+\alpha_3+\alpha_4$ and $\beta:=\alpha_2$
(the simple roots of the root system
$G_2=(\Phi_\sigma,\Phi_\sigma^\vee)$) in both cases $X^\sigma=P^\sigma=Q^\sigma$
has basis $\alpha,\beta$. The
highest root of the root system $G_2$ is $2\alpha+3\beta$, in both cases
$Y_\sigma$ has basis the simple coroots of $G_2$ given by
$\alpha^\vee=(\alpha_1^\vee+\alpha_3^\vee+\alpha_4^\vee)/3$ and
$\beta^\vee=\alpha^\vee_2$ and the fundamental coweights are
$\varpi_\alpha^\vee=2\alpha^\vee+\beta^\vee$ and
$\varpi_\beta^\vee=3\alpha^\vee+2\beta^\vee$. We have $\tW=\Wa$ and
the extremal
points of the fundamental alcove are given by 0, $\alpha^\vee+\frac12\beta^\vee$
and $\alpha^\vee+\frac23\beta^\vee$. By Corollary \ref{root system of Gtso2},
the corresponding root systems
of $\Gtso$ are respectively of type $G_2$, $A_1\times A_1$ and $A_2$
--- see also the table at the end of the paper.
\end{example}
The following remark will be useful
\begin{proposition}\label{cas facile}
When $\bG$ is quasi-simple, $n_\sigma>1$ and there are no special orbits, 
the root systems $\Sigma_{t\sigma}$
and $\Phi_{t\sigma}$ are dual to each other.
\end{proposition}
\begin{proof}
This is an immediate consequence of the definitions: when $n_\sigma>1$
then $\bG$ is of type $ADE$ thus $\Sigma\simeq\Sigma^\vee$,
and via this isomorphism the definitions of $\Phi_{t\sigma}$ and 
$\Sigma_{t\sigma}^\vee$ match, as well as those of $\Phi_{t\sigma}^\vee$
and $\Sigma_{t\sigma}$.
\end{proof}
\section{Isolated and quasi-isolated elements}
\begin{definition}
We say that the \qss\ element $t\sigma$ is {\em isolated}
\index{isolated}
if $\Gtso$ is not contained in a $\sigma$-stable Levi subgroup of 
a proper $\sigma$-stable parabolic subgroup of $\bG$.
\end{definition}
\begin{proposition}\label{isolated}The following are equivalent:
\begin{enumerate}
\item $t\sigma$ is isolated.
\item Every central torus in $\Gtso$ is central in $\bG$.
\item $Z^0(\Gtso)=Z^0(\Gso)$.
\item $\Sigma_{t\sigma}$ and $\Sigma_\sigma$ span the same subspace of
$X(\Tso)\otimes \BQ$.
\item[(4')] $\Phi_{t\sigma}$ and  $\Phi_\sigma$ span the same subspace of
$X(\Tso)\otimes \BQ$.
\item $W^0(t\sigma)$ is not in a proper parabolic subgroup of $W^\sigma$.
\end{enumerate}
\end{proposition}
\begin{proof}
Let us prove (1)$\Leftrightarrow$(2). If (1) fails then $\Gtso$ is
contained   in  a   proper  $\sigma$-stable   Levi  subgroup   $\bL$  of  a
$\sigma$-stable   parabolic   subgroup   of   $\bG$.   In   particular,  by
\cite[1.25]{grnc}   $\Lso$  is  a  proper   Levi  subgroup  of  $\Gso$  and
$Z^0(\Lso)$  is  central  in $\bL$ but not in $\bG$ by
\cite[1.23]{grnc}.  By \cite[1.8(iii)]{grnc}  $\Tso=(\bT^{t\sigma})^0$ is a
maximal  torus  of  $\Gtso$,  thus  it  is  a maximal torus of $\Lso$, and
$Z^0(\Lso)$ is in this torus thus in $\Gtso$; and $Z^0(\Lso)$ is central in
$\Gtso$ since it is central in $\bL$. Thus (2) fails.

If (2) fails, then there is a torus $\bS$ central in $\Gtso$ not central in
$\bG$;  the torus $\bS$ is in all maximal tori of $\Gtso$, in particular in
$\Tso$ thus $\bS\subset\Gso$; thus by \cite[1.25(ii)]{grnc},
$C_\bG(\bS)$ is a Levi subgroup of
a $\sigma$-stable proper parabolic subgroup of $\bG$ containing $\Gtso$.

Let us prove (2)$\Leftrightarrow$(3). We first observe that one has always
$Z^0(\Gtso)\supset Z^0(\Gso)$ since by Proposition \ref{Gs ss} we have
$Z^0(\Gso)=((Z\bG)^{\sigma})^0$, and since $Z\bG\subset\bT$, we have
$((Z\bG)^{\sigma})^0)=((Z\bG)^{t\sigma})^0$.
Now if (2) holds every element of $Z^0(\Gtso)$ is central in $\bG$ and
contained in $\Tso$, thus
central in $\Gso$, whence equality. Conversely, if there is equality any
torus central in $\Gtso$ is central in $\Gso$, thus in $\bG$ since
$Z^0(\Gso)\subset Z^0(\bG)$ by \cite[1.23]{grnc} applied with $\bL=\bG$.

Let us prove (3) $\Leftrightarrow$ (4).  It is a general fact in a reductive
group $\bG$ that if $\Sigma^\perp$ is the orthogonal of $\Sigma$ in 
$Y\otimes\BQ$ then $Y(Z^0\bG)=Y\cap \Sigma^\perp$ and generates $\Sigma^\perp$.
Applying this to $\Gtso$, (3) can be rewritten
$\Sigma_{t\sigma}^\perp=\Sigma_\sigma^\perp$ where $\perp$ is taken in
$Y(\Tso)\otimes\BQ$.
We get (4) by taking the orthogonal of this equality.

(4) is equivalent to (4') since $\Phi_{t\sigma}$ and $\Sigma_{t\sigma}$ 
span the same $\BQ$-vector space,
as do $\Phi_{\sigma}$ and $\Sigma_{\sigma}$.

We finally prove (4) $\Leftrightarrow$ (5).
We have $W^\sigma=W^0(\sigma)$ by Proposition \ref{Wsigma=Wosigma}. The equivalence
of (4) and (5) thus follows from the fact that
$W^0(t\sigma)$ and $W^0(\sigma)$ are the Weyl groups of the root systems
$\Sigma_{t\sigma}$ and $\Sigma_\sigma$ respectively.
\end{proof}
\begin{corollary}
If $\bG$ is semisimple and $t\sigma$ isolated, then $\Gts$ is semisimple.
\end{corollary}
\begin{proof}
This is a straightforward consequence of Proposition \ref{isolated}(2).
\end{proof}
\begin{remark} Proposition \ref{isolated}(3) fails if one does not take the
connected components, that is
we do not have in general $Z(\bG^\sigma)\subset Z(\Gts)$,
as the following example shows:
take $\bG=\GL_3$ with $\sigma$ of symplectic type and $t$ such that 
$t\sigma$ is of orthogonal type; then $\Gs\simeq \Sp_2\times\{\pm1\}$ and $\Gts\simeq\Orth_3$;
we have $Z(\Gs)\simeq(\BZ/2\BZ)^2$ and $Z(\Gts)\simeq\BZ/2\BZ$.
\end{remark}
To study isolated classes we can assume $\bG$ to be semisimple, thanks to the
following proposition.
\begin{proposition}
We have $\Tso=((Z\bG)^\sigma)^0((\bT\cap D\bG)^\sigma)^0$ and if we write
$t\in\Tso$ as $t=zt_1$ with $z\in((Z\bG)^\sigma)^0$
and $t_1\in((\bT\cap D\bG)^\sigma)^0$ then  $\bG^{t\sigma}=\bG^{t_1\sigma}$ and
$t\sigma$ is isolated in $\bG\cdot\sigma$ if and only if 
$t_1\sigma$ is isolated in $D\bG\cdot\sigma$.
\end{proposition}
\begin{proof}
We claim that \begin{equation}
\label{Gtso=}\Gtso=((Z\bG)^\sigma)^0((D\bG)^{t\sigma})^0 .
\end{equation}
Indeed, by 
\cite[1.31]{grnc} the index of
$(Z\bG)^\sigma(D\bG)^{t\sigma}$ in $\Gts$ is finite so that the identity components
coincide.
Since $((Z\bG)^\sigma)^0\subset \Tso$, we deduce  that any $t\in\Tso$ 
can be written $t=zt_1$
with $z\in((Z\bG)^\sigma)^0$ and $t_1\in((\bT\cap D\bG)^\sigma)^0$. We then have
$\bG^{t\sigma}=\bG^{t_1\sigma}$ since $z$ commutes with $\bG$ and with $\sigma$.
Hence $t\sigma$ is isolated in $\bG\cdot\sigma$ if and only if $t_1\sigma$ is isolated in 
$\bG\cdot\sigma$ which is equivalent to $t_1\sigma$ being isolated in
$D\bG\cdot\sigma$ since by 
Equation \ref{Gtso=} applied with $t_1$ instead of $t$ the Weyl group 
$W^0(t_1\sigma)$ is the same in $\bG$ and in $D\bG$.
\end{proof}
\begin{proposition} \label{Stsigma}
If the image of $t\sigma$ in $\LT$ is parameterized
by $\lambda\in\CC$, and $S_{t\sigma}$ is the set of reflections with respect
\index{Sts@$S_{t\sigma}$}
to the walls of $\CC$ containing $\lambda$, then $(W^0(t\sigma),S_{t\sigma})$
is a Coxeter system.
 The subset $\Delta_{t\sigma}$
\index{Dts@$\Delta_{t\sigma}$}
of $\tilde\Delta_\sigma$ 
corresponding to the reflections in $S_{t\sigma}$
is a basis of $\Phi_{t\sigma}$.
\end{proposition}
\begin{proof}
If $\lambda\in\CC$,
by Proposition \ref{Phi_tsigma} and
\cite[Chap. V \S3 no 3, (vii) of Proposition 1]{Bou}, a set of
generators of $W^0(t\sigma)$  is given by the reflections with respect to
the walls of $\CC$ containing $\lambda$. The set $S_{t\sigma}$ of these 
reflections is a subset of $\tilde S_\sigma$, thus
$(W^0(t\sigma),S_{t\sigma})$ is a parabolic subsystem of $(\tilde
S_\sigma,\Wa)$.

To see that $\Delta_{t\sigma}$ is a basis of
$\Phi_{t\sigma}$, we note that for any subset of $\tilde S_\sigma$
a basis of the corresponding parabolic subsystem is given by the corresponding
subset of $\tilde\Delta_\sigma$.
This is a consequence of 
the following lemma
\begin{lemma}
Let  $\Delta$ be a basis  of a root system  $\Phi$ and let $\alpha_0$ be the
highest   root  of   $\Phi$;  then   for  any   $\alpha\in\Delta$  the  set
$\Delta_\alpha:=(\Delta\cup\{-\alpha_0\})-\{\alpha\}$  is  a  basis  of  the
root subsystem it generates.
\end{lemma}
\begin{proof}Let  $\Delta:=\{\alpha_1,\ldots,\alpha_r\}$  with
$\alpha=\alpha_1$ and let $\alpha_0=\sum_i n_i\alpha_i$.
Let   $\beta\in\Phi^+$   be  of the form
$\sum_{\alpha_i\in\Delta_\alpha}m_i\alpha_i$ with $m_i\in\BZ_{\ge 0}$. 
We have
$\beta=-\frac{m_1}{n_1}(-\alpha_0)+\sum_{i\geq2}(m_i-n_i\frac{m_1}{n_1})\alpha_i$.
Since  $\alpha_0$ is the highest root we have $0\leq\frac{m_1}{n_1}\leq 1$.
As  $\frac{m_1}{n_1}$ is an integer, it is equal to 0 or 1. If $m_1=n_1$ we
have   $\beta=-(-\alpha_0)+\sum_{i\geq2}(m_i-n_i)\alpha_i$   and   all  the
coefficients  are  non-positive.  If  $m_1=0$  all  coefficients  are  non
negative.
\end{proof}
\end{proof}
\begin{proposition} \label{param isolated} Assume
$\bG^\sigma$ semisimple, so that every conjugacy class of \qss\ elements in
$\bG\cdot\sigma$ has a representative parameterized by $\lambda\in\CC$.

If  $t\sigma$ is such a representative, then
$t\sigma$  is   isolated  if  and  only  if
$|\tilde S_{\sigma,i}-S_{t\sigma}|=1$
for each irreducible component $S_{\sigma,i}$ of $S_\sigma$.

Conversely, if $\Omega$ is a subset of
$(\tilde S_\sigma)_{p'}:=\{s\in\tilde S_\sigma\mid \frac 1{n_s}
\varpi^\vee_s\in Y_\sigma\otimes\BQ_{p'}\}$
such that $|\Omega\cap\tilde S_{\sigma,i}|=1$ for each $i$, let
$\{s_i\}=\tilde S_{\sigma,i}\cap\Omega$, and let $\varpi^\vee_{s_i}$ be
the corresponding fundamental coweight
(where by convention $\varpi_{s_{0,i}}^\vee=0$).
Then $\sum_i\frac1{n_{s_i}}\varpi_{s_i}^\vee$ is the unique element
of $\CC$ such that for a corresponding $t\sigma$ we have
$S_{t\sigma}=\cup_i (\tilde S_{\sigma,i}-\{\alpha_{s_i}\})$.

It follows that the $\bG$-conjugacy classes of isolated elements
of finite order of $\bG\cdot\sigma$ are
parameterized by the $\tW/\Wa$-orbits of 
subsets $\Omega$ of $(\tilde S_\sigma)_{p'}$ such that 
$|\Omega\cap\tilde S_{\sigma,i}|=1$ for each $i$; the action of
$\tW/\Wa$ comes from the embedding $\tW\subset\Wa'$ --- which exists
since $\Gs$ is semisimple ---
and from the identification of $\Wa'/\Wa$ to a group of automorphisms of $\CC$,
see Proposition \ref{aff}(5).
\end{proposition}
\begin{proof}
By Proposition \ref{Stsigma}, $(W^0(t\sigma),S_{t\sigma})$ is a Coxeter system.
The class of $t\sigma$ is isolated if and only if $S_{t\sigma}$ has same rank as
$W$, that is if and only if, setting $\Omega:=\tilde S-S_{t\sigma}$, we have 
$|\Omega\cap\tilde S_{\sigma,i}|=1$
for each irreducible component $S_{\sigma,i}$ of $S_\sigma$.

Conversely such a set $\Omega\subset\tilde S_\sigma$ defines a unique point
$\lambda\in\CC$ namely the intersection of the walls  of $\CC$ corresponding 
to $\tilde S-\Omega$: indeed,
since in each irreducible component we consider all the walls but one,
such a $\lambda$ is the product of extremal points of the simplices $\CC_i$,
hence, by Proposition \ref{aff}(2), is equal to $\sum_i\frac1{n_{s_i}}\varpi_{s_i}^\vee$.
\end{proof}
We will need the following lemma in the proof of the next theorem.
\begin{lemma} \label{centerquotient}
If $\bH$ is a connected reductive group and $\bK$ is a closed
subgroup of $Z\bH$, then $Z(\bH/\bK)=(Z\bH)/\bK$.
\end{lemma}
\begin{proof}Let $\bT$ be a maximal torus of $\bH$.
The roots of $\bH$ are the inflations to $\bT$ 
of the roots of $\bH/\bK$ with respect to $\bT/\bK$ since the root subgroups 
of $\bH$ are mapped isomorphically to the root subgroups of $\bH/\bK$.
The center of $\bH/\bK$ is the intersection of the kernels of the roots, and
$\bK$ lies in that intersection, whence the result.
\end{proof}

In what follows, we write $e(G)$ for the exponent of a group $G$ and
\index{eG@$e(G)$}
$e(x)$ for the order of $x\in G$. For an algebraic group $\bG$,
\index{ex@$e(x)$}
we set $\AZ(\bG):=Z(\bG)/Z^0(\bG)$.
\index{AZ@$\AZ$}
\begin{theorem}\label{A(ZGtso)} As in Context \ref{fix qss},
let $\sigma$ be quasi-central inducing an automorphism of order
$n_\sigma$ and let $t\in\Tso$ be of finite order $e(t)$;
if $Z\bG$ is connected then $e(\AZ(\Gtso))$ divides $n_\sigma e(t)$.
\end{theorem}
\begin{proof}
We prove this by a series of reductions.

$\bullet$ The first step is to reduce to the case where $t\sigma$ is isolated.
Otherwise $\Gtso$ is in a proper Levi subgroup of $\bG$ and we may replace
$\bG$ with this proper Levi subgroup without changing $\Gtso$ or $t\sigma$.

$\bullet$ The next step is to assume that $Z\bG=[Z\bG,\sigma]$.
Indeed we may quotient $\bG$ by $((Z\bG)^\sigma)^0\subset Z^0(\Gtso)$
which does not change $\AZ(\Gtso)$ nor $n_\sigma $ and can only replace $e(t)$ with a
divisor. The center of $\bG/((Z\bG)^\sigma)^0$ is
$Z':=Z\bG/((Z\bG)^\sigma)^0$  by  Lemma  \ref{centerquotient}  and satisfies
$Z'=[Z',\sigma]$ since $Z\bG=((Z\bG)^\sigma)^0[Z\bG,\sigma]$.

$\bullet$  The  next  step  is  to  reduce  to  the  case $\bG$ adjoint. If
$\bG_\ad=\bG/Z\bG$ then we have an exact sequence $1\to  (Z\bG)^\sigma\to
\Gts\to\bG_\ad^{t\sigma}\to  1$;  indeed  the  only non-trivial point is the
surjectivity;  if  $\bar  g\in\bG_\ad^{t\sigma}$ then it is the image  of
$g\in\bG$  such that  $g\inv\lexp{t\sigma}g\in Z\bG$,  and by  the equality
$Z\bG=[Z\bG,\sigma]$ we may write $g\inv
\lexp{t\sigma}g=z\inv\lexp{t\sigma}z$  with  $z\in  Z\bG$,  thus $gz\inv\in
\Gts$ has same image in $\bG_\ad$ as $g$.

Taking the identity components, the above exact sequence gives the exact
sequence 
$1\to Z\bG\cap\Gtso\to \Gtso\to (\bG_\ad^{t\sigma})^0\to 1$.
Taking centers, we get by Lemma \ref{centerquotient} the exact sequence
$1\to  Z\bG\cap\Gtso\to  Z(\Gtso)\to Z((\bG_\ad^{t\sigma})^0)\to 1$.

Since we have $(Z\bG\cap\Gtso)^0=((Z\bG)^{t\sigma})^0=((Z\bG)^\sigma)^0$,
the exponent of $(Z\bG\cap\Gtso)/(Z\bG\cap\Gtso)^0$ divides that
of $(Z\bG)^\sigma/((Z\bG)^\sigma)^0$,
and the exponent of this last group divides $n_\sigma $ by
Lemma \ref{A(Ts)}; it follows that $e(\AZ\Gtso)$ divides the lcm
of  $n_\sigma $ and $e(\AZ((\bG_\ad^{t\sigma})^0))$.
It is thus sufficient to prove the theorem for $\bG_\ad$.

$\bullet$ The next step is to reduce to the case $\bG$ simple. Indeed,
since $\bG$ is adjoint, it is uniquely a direct product of simple groups,
permuted by $\sigma$. If we consider a subset of $i$ components
cyclically permuted by $\sigma$, the group of fixed points under $t\sigma$ is
isomorphic to the group of fixed points under $(t\sigma)^i$ in one component.
Since $n_\sigma e(t)$ (resp.~$e(\AZ\Gtso)$) is the lcm of the analogous numbers
for each component, it is enough to show the theorem for a simple group.

$\bullet$ For a simple group we use the description of isolated elements in
Proposition \ref{param  isolated}; thus $t\sigma$ has same centralizer as an element of
finite  order  whose  image  $\overline  t$  in  $\LT$  is parameterized by
$\lambda=\frac1{n_s}\varpi_s^\vee\in    Y_\sigma\otimes\bQ_{p'}$   where
$\Omega=\{s\}$. If  $z$  is
central  in  $\Gtso$ then $(\bG^{zt\sigma})^0\supseteq\Gtso$, 
thus  $W^0(t\sigma)\subseteq W^0(zt\sigma)$. By 
Proposition \ref{Phi_tsigma} we deduce $\Phi_{t\sigma}\subseteq\Phi_{zt\sigma}$.
Since $\Phi_{t\sigma}$ (resp.~$\Phi_{zt\sigma}$)
consists of  the roots  of $\Phi_\sigma$ such that $\alpha(\lambda)\in\BZ$
(resp.~$\alpha(\mu)\in\BZ$ if $\mu$ parameterizes the  image of $tz$ in $\LT$),
we have  $\alpha(\mu)\in\BZ$ if $\alpha(\lambda)\in\BZ$.
It   follows   that  $\mu=\frac  a{n_s}\varpi_s^\vee$  modulo
$Y_\sigma$  for some $a\in\BZ$, thus the image $\overline {tz}$
of $tz$ in $\LT$ is a power of $\overline
t$. Thus $e(\overline z)$ divides $e(\overline t)$.

Since the kernel of the map $\Tso\to\LT$ has exponent dividing $n_\sigma $
by Lemma \ref{A(Ts)}, it follows that $e(z)$ divides $n_\sigma  e(t)$.
\end{proof}
\subsection*{Quasi-isolated elements}
\begin{definition}
We say that the \qss\ element $t\sigma$ is {\em quasi-isolated}
\index{quasi-isolated}
if $\Gts$ is not contained in a $\sigma$-stable Levi subgroup of 
a proper $\sigma$-stable parabolic subgroup of $\bG$.
\end{definition}
In particular an isolated element is quasi-isolated.
\begin{proposition}\label{quasi-isolated} The following are equivalent:
\begin{enumerate}
\item $t\sigma$ is quasi-isolated.
\item $W(t\sigma)$ is not in a proper parabolic subgroup of $W^\sigma$.
\end{enumerate}
\end{proposition}
\begin{proof}
The parabolic subgroups of the Weyl group of a
reductive group are in one-to-one correspondence with the Levi subgroups
containing a given maximal torus. Applying this in $\Gso$ we deduce from
\cite[1.25(ii)]{grnc} that parabolic subgroups of $W^\sigma$ are in
one-to-one correspondence with the $\sigma$-stable Levi subgroups
containing $\bT$ of $\sigma$-stable parabolic subgroups of $\bG$.

We prove that (1) implies (2). If (2) fails,
by the above $W(t\sigma)$
has representatives in a $\sigma$-stable Levi
subgroup $\bL$ of a proper $\sigma$-stable parabolic subgroup.
By the proofs of Proposition \ref{AG(ts)} and Proposition \ref{exactsequence},
$\Gts$ is generated by $\Gtso$ and representatives of
$W(t\sigma)$. By the equivalence of (1) and (5) in Proposition \ref{isolated}
the group $\Gtso$ is
contained in the Levi subgroup $\bL$, 
thus the whole of $\Gts$ is contained
in a proper Levi subgroup, that is $t\sigma$ is not quasi-isolated.

We prove that (2) implies (1). If (1) fails then $\Gts$ is
contained in a $\sigma$-stable Levi subgroup $\bL$ of a proper
$\sigma$-stable parabolic subgroup, hence representatives of all elements of
$W(t\sigma)$ can be taken in $\bL$, thus $W(t\sigma)$ is contained in the
Weyl group of $\bL$ and (2) fails.
\end{proof}
\begin{proposition}\label{W(t sigma)}
If the class of $t\sigma$ is parameterized by $\lambda$ (see Theorem
\ref{parametrization})
then $W(t\sigma)=\{w\in W^\sigma\mid w(\lambda)-\lambda\in Y_\sigma\}$.
\end{proposition}
\begin{proof} The class of $t\sigma$ being parameterized by 
$\lambda\in Y_\sigma\otimes\BQ$ means that the image $\bar t$ of $t$
in $\LT$ is identified with the image $\bar\lambda$ of $\lambda$ in 
$Y_\sigma\otimes\BQ/\BZ$
(see the proof of Theorem \ref{parametrization}); hence $w\in W^\sigma$ fixes 
$\bar t$
if and only if it fixes $\bar \lambda$, whence the assertion.
\end{proof}
\begin{remark}
The  following example shows that an element quasi-isolated in $\bG\cdot\sigma$
may not be quasi-isolated in $D(\bG)\cdot\sigma$, which forbids taking
the derived group as a method to reduce the study of quasi-isolated elements
to the case of semisimple groups. Let $\bG=\GL_2$, thus
$D(\bG)=\SL_2$.  Let  us  take  $\sigma$ quasi-central defined on
$\bG$ by $\sigma(x)=w\lexp t(x\inv)w\inv$ where $w=\begin{pmatrix}0&1\\-1&0
\end{pmatrix}$. We have $\Gso=D(\bG)$. Let
$t=\begin{pmatrix}i&0\\0&-i\end{pmatrix}$.   Then   $W^0(t\sigma)=1$,   and
$W(t\sigma)=1$    in    $D(\bG)$    but    $W(t\sigma)=W^\sigma=W$
in  $\bG$.  The  centralizer of $t\sigma$ in
$D(\bG)=\SL_2$    is   the   maximal torus $\bT'=\diag(a,a\inv)$, 
but   $\Gts$   is   generated   by   $\bT'$   and
$\begin{pmatrix}0&1\\1&0\end{pmatrix}$. Thus $t\sigma$ is quasi-isolated in
$\bG\cdot\sigma$ but not in $D(\bG)\cdot\sigma$. The ``explanation'' is that $Y_\sigma$ is generated
by $\alpha^\vee/2$ in the $\GL_2$ case, and by $\alpha^\vee$ in the
$\SL_2$ case, so applying Proposition \ref{W(t sigma)} yields different results
--- equivalently $[\bT,\sigma]$ is trivial in $\SL_2$ but is equal to the center in $\GL_2$,
so using the definition of $W(t\sigma)$ one gets different results.
\end{remark}
\begin{definition}
We define $\CA_{t\sigma}:=\{w\in W(t\sigma)\mid
w(\Delta_{t\sigma})=\Delta_{t\sigma}\}$.
\index{Ats@$\CA_{t\sigma}$}
\end{definition}
\begin{proposition}\label{W/W0}
$W(t\sigma)=W^0(t\sigma)\rtimes\CA_{t\sigma}$.
\end{proposition}
\begin{proof}
By  Proposition \ref{Stsigma},  $\Delta_{t\sigma}$ is  a basis  of the root
system of $W^0(t\sigma)$. Since
two  such bases are  conjugate under $W^0(t\sigma)$, and $W^0(t\sigma)$ is the
(normal)  subgroup  generated  by  the  reflections  of  $W(t\sigma)$  (see
Proposition  \ref{Phi_tsigma} and Proposition \ref{W(t  sigma)}) we get the
semi-direct decomposition.
\end{proof}
\begin{proposition}\label{Y'_sigma embeds}
The group $\CA_{t\sigma}$ embeds in
$\CA_{R(\sigma)}:=(Y_\sigma\cap\BQ\Phi_\sigma^\vee)/Q(\Phi_\sigma^\vee)$.
\index{ARs@$\CA_{R(\sigma)}$}
\end{proposition}
\begin{proof} 
Consider the map $W(t\sigma)\to Y_\sigma/Q(\Phi_\sigma^\vee):
w\mapsto w(\lambda)-\lambda \pmod{Q(\Phi_\sigma^\vee)}$ where $t\sigma$
is parameterized by $\lambda$, see Theorem \ref{parametrization}. Its image lies in
$(Y_\sigma\cap\BQ\Phi_\sigma^\vee)/Q(\Phi_\sigma^\vee)$ since $w(\lambda)-\lambda$ is a 
$\BQ$-linear combination of coroots.
Since   $W^\sigma$  is  generated  by  the  reflections  with  respect  to
$\Phi_\sigma$  by Proposition \ref{Phi_sigma}, using the formula defining a
reflection we see that $W^\sigma$ acts trivially on $Y_\sigma/Q(\Phi_\sigma^\vee)$,
hence the above map is a group morphism.
Its kernel is
$\{w\in W^\sigma\mid w(\lambda)-\lambda\in Q(\Phi_\sigma^\vee)\}$ which is a 
reflection group by \cite[Ch.\ VI ex.\ 1 of \S 2]{Bou}, hence is 
generated by the reflections satisfying the same condition, thus 
is equal to $W^0(t\sigma)$ by Proposition \ref{Phi_tsigma}.
\end{proof}
Note that by Proposition \ref{aff}(5), $\CA_{R(\sigma)}$ can be identified to a subgroup
of $\Aut_{W^\sigma}(\tilde\Delta_\sigma)$.
\begin{remark} \label{R(sigma) simply connected}
In  view of Proposition  \ref{W/W0} and  Proposition \ref{Y'_sigma embeds},  if $\CA_{R(\sigma)}=1$
(which  can be stated as ``the derived  root datum of $R(\sigma)$ is simply
connected'')   then $W(t\sigma)=W^0(t\sigma)$  and
every   quasi-isolated   element   is  isolated.
\end{remark}
\begin{proposition} \label{car qisole}
Let $t\sigma$ be an element parameterized by 
$\lambda\in\CC$, then $t\sigma$ is quasi-isolated if and only if 
$\CA_{t\sigma}$ acts transitively on $\tilde S_{\sigma,i}-S_{t\sigma}$ 
for every $i$. 

For such an element, $\CA_{t\sigma}$ is the stabilizer of 
$\Delta_{t\sigma}$ in $\CA_{R(\sigma)}$.
\end{proposition}
\begin{proof}
We show that we can  apply \cite[Corollary 4.3(b)]{cedric}. \cite{cedric}
considers the Weyl group $W$ of a root system $\Phi$ acting on a lattice
$Y(\bT)$ and on the space $V=Y(\bT)\otimes\BQ$, the affine group 
$W_{\text{aff}}=W\rtimes\BZ\Phi$, the fundamental domain $\CC$ of
$W_{\text{aff}}$
on $V$ and an element $\lambda\in\CC$. We can specialize these objects 
to correspond to our setting as follow:
$$\begin{array}{ccccccc}
W&\Phi&Y(\bT)&V&W_{\text{aff}}&\CC&\lambda\\
W^\sigma&\Phi_\sigma&Y_\sigma\cap\BQ\Phi_\sigma^\vee&\BQ\Phi_\sigma^\vee&\Wa&\CC&\lambda\\
\end{array} $$
Then, comparing our  Proposition \ref{W(t sigma)},
Proposition \ref{Phi_tsigma} to the definitions
\cite[below Lemma 3.1]{cedric}, we can extend this dictionary as follows 
$$\begin{array}{cccc}
W_\bG(\lambda)&W^0(\lambda)&\Phi(\lambda)\\
W(t\sigma)&W^0(t\sigma)&\Phi_{t\sigma}\\
\end{array}$$
Then comparing our Proposition \ref{Stsigma} to \cite[proof of 3.5]{cedric},
our Proposition \ref{Y'_sigma embeds} to \cite[3.C]{cedric}, and our Proposition
\ref{Stsigma} to \cite[Proposition 3.14]{cedric} we have the correspondences
$$\begin{array}{ccc}
\tilde\Delta&\CA_\bG&I_\lambda\\
\tilde\Delta_\sigma&\CA_{R(\sigma)}&\Delta_{t\sigma}\\
\end{array}$$
Now by \cite[below Lemma 3.1]{cedric} joint with the choice made
in \cite[last line before section 4]{cedric} $A_\bG(\lambda)$ of {\it loc.~cit.}
is our $\CA_{t\sigma}$.

We could have equivalently given the first sentence of Proposition \ref{car qisole}
in terms of the action of $\CA_{t\sigma}$ on $\Delta_{t\sigma}$,
the condition for a quasi-semisimple element to be quasi-isolated given
in Proposition \ref{quasi-isolated} becoming that given below Lemma
3.2 in \cite{cedric}, we can use \cite[Corollary 4.3 (b)]{cedric} which gives
the first sentence of Proposition \ref{car qisole}.

To  prove the last sentence of Proposition \ref{car qisole}, following \cite{cedric} we
introduce  affine  coordinates $\{\lambda_\alpha\}_{\alpha\in\tilde\Delta}$
for $\lambda$ as in \cite[below (3.9)]{cedric}. Then by
\cite[3.14(b)]{cedric}   we  have  $$\CA_{t\sigma}=\{a\in\CA_{R(\sigma)}\mid  \forall
\alpha\in\tilde\Delta,      \lambda_{a(\alpha)}=\lambda_\alpha\}.\eqno(*)$$
Since  by  the  first  sentence  of Proposition  \ref{car  qisole} $\CA_{t\sigma}$ acts
transitively    on   $\tilde\Delta-\Delta_{t\sigma}$,   it   follows   that
$\lambda_\alpha$  is constant  (and non-zero  by \cite[3.14(a)]{cedric}) on
this   set.   Since   $\lambda_\alpha$   is   constant   equal  to  $0$  on
$\Delta_{t\sigma}$   (by  \cite[3.14(a)]{cedric}),  it   is  necessary  and
sufficient to stabilize $\Delta_{t\sigma}$ in order to satisfy the condition
on  the  right-hand  side  of $(*)$,  whence the second sentence of Proposition \ref{car
qisole}.
\end{proof}
\begin{proposition} \label{param quasi-isolated} Assume
$\bG^\sigma$ semisimple, so that every conjugacy class of \qss\ elements in
$\bG\cdot\sigma$ has a representative parameterized by $\lambda\in\CC$. Then:

If $\Omega$ is a subset of
$(\tilde S_\sigma)_{p'}:=\{s\in\tilde S_\sigma\mid \frac 1{n_s}
\varpi^\vee_s\in \BZ_{(p)}\Phi_\sigma^\vee\}$
such that the stabilizer in $\CA_{R(\sigma)}$ of
$\Omega\cap\tilde S_{\sigma,i}$  acts transitively on it for each $i$,
the element $\sum_i\dfrac{\sum_{s\in\Omega\cap \tilde S_i}\varpi_s^\vee}
{n_i(\Omega)|\Omega\cap\Delta_{\sigma,i}|}$, where $n_i$ is the common value
of $n_s$ for every $s\in\Omega\cap \tilde S_i$, is the unique point
of $\CC$ such that for a corresponding $t\sigma$ we have
$S_{t\sigma}=\cup_i (\tilde S_{\sigma,i}-\Omega)$ and $\CA_{t\sigma}$
acts transitively on  $\tilde S_{\sigma,i}-S_{t\sigma}$.

The $\bG$-conjugacy classes of quasi-isolated elements
of finite order of $\bG\cdot\sigma$ are
parameterized by the $\tW/\Wa$-orbits of 
subsets $\Omega$ of $(\tilde S_\sigma)_{p'}$ such that 
the stabilizer in $\CA_{R(\sigma)}$ of $\Omega\cap\tilde S_{\sigma,i}$
acts transitively on it for each $i$.
\end{proposition}
\begin{proof}
Given the dictionary between our setting and \cite{cedric} in the proof of
Proposition \ref{car qisole}, the first
assertion of the proposition is a translation to our setting of 
\cite[4.B]{cedric}. 

The second assertion of the proposition is a consequence of the first one
and of Proposition \ref{car qisole}.
\end{proof}

\section{Classification of quasi-isolated elements}
The previous propositions make up together
a method to classify quasi-isolated elements and their centralizers,
which can be described as follows:
\begin{itemize}
\item Describe the root datum $R(\sigma)$.
\item Describe the quasi-isolated classes associated to this root datum using
\cite{cedric}; by Proposition \ref{param quasi-isolated} these classes lift to
quasi-isolated classes in $\bG\cdot\sigma$.
\item Determine $\Ts/\Tso$ using Lemma \ref{A(Ts)}.
\item Determine the type of the root datum of $\Gtso$ either by Proposition \ref{cas
facile} or an explicit computation when there are special orbits.
\item Determine the isogeny type of $\Gtso$ by computing its center
explicitly.
\item Determine $W(t\sigma)$ using that it can be computed in the datum
$R(\sigma)$, see Proposition \ref{W(t sigma)}.
\end{itemize}
We follow this method in the next subsections.
\subsection*{Quasi-isolated classes in type $A$, $n_\sigma=2$} 
We  give  the  classification  of  quasi-isolated  and  isolated
classes   of   finite   order   in   $\GL_r.\sigma$,   $\SL_r.\sigma$   and
$\PGL_r.\sigma$, where $\sigma$ is the quasi-central automorphism defined by
$\sigma(g)= J\lexp t(g\inv)J\inv$ where 
$J:=\antidiag(\underbrace{1,\ldots,1}_{\lfloor\frac r 2\rfloor},
\underbrace{-1,\ldots,-1}_{\lfloor\frac {r+1} 2\rfloor})$;
here $\antidiag(x_1,\ldots,x_r)$ denotes the antidiagonal matrix with 
coefficient $x_i$ in the $i$-th column.
Note that for $r=2$ the automorphism $\sigma$ acts trivially on
$\SL_2$ and $\PGL_2$, while still defining on $\GL_2$ a conjugacy class of
non-inner quasi-central automorphisms.

\begin{proposition}
If  $p=2$  the class of $\sigma$ is the only
quasi-isolated class in $\GL_r.\sigma$, $\SL_r.\sigma$ or $\PGL_r.\sigma$.
In  these groups the points over $k$ of the 
centralizer of $\sigma$  is $\Sp_r(k)$ if $r$ is even and
$\Sp_{r-1}(k)$ if $r$ is odd.
\end{proposition}
\begin{proof}
The  root  datum  $R(\sigma)$  is  of  type $B_{\lfloor\frac r2\rfloor}$ or
$C_{\lfloor\frac  r2\rfloor}$. For  these types  if $p=2$  one has $(\tilde
S_\sigma)_{p'}=\{\varpi^\vee_{s_0}\}$,   whence  the   first  assertion  by
Proposition  \ref{param  quasi-isolated}.  The  second assertion comes from
Proposition  \ref{type of Gso}(5)  and (6), using  that in characteristic 2
all semisimple groups of type $C_{\lfloor\frac r2\rfloor}$ have same set of
points over a field,
and  that  in  characteristic  2,  in  type $A_{2n}$, $\sigma$ is always of
symplectic type.
\end{proof}

\begin{proposition}
Assume $p\neq 2$; let $i$ be a primitive fourth root of unity,
and for $j=0,\ldots,\lfloor\frac r2\rfloor$ let
$t(j):=\diag(\underbrace{i,\ldots,i}_j,1,\ldots,1,\underbrace{-i,\ldots,-i}_j)
\in\SL_r$.
Then:
\begin{enumerate}
\item  The elements $t(j).\sigma$ with
$j=0,\ldots,\lfloor\frac    r2\rfloor$    are    representatives   of   the
quasi-isolated conjugacy   classes  in
$\GL_r.\sigma$,  $\PGL_r.\sigma$ for all $r$ and in $\SL_r.\sigma$ for odd
$r$.  They are all isolated,
except in $\GL_r$ and $\PGL_r$ when $r$ is even and $j=1$.

For odd $r$ and $\bG=\SL_r$  or  $\PGL_r$ (resp. $\bG=\GL_r$)
we have $C_\bG(t(j).\sigma)=\SO_{2j+1}\times\Sp_{r-1-2j}$
(resp.~$C_\bG(t(j).\sigma)=\Orth_{2j+1}\times\Sp_{r-1-2j}$).

For even $r$ and $\bG=\GL_r$ (resp.~$\bG=\PGL_r$)
we have $C_\bG(t(j).\sigma)=\Orth_{2j}\times\Sp_{r-2j}$
(resp.~$(\Orth_{2j}\times\Sp_{r-2j})/\pm1$).
\item  The  elements
$t(j).\sigma$ with  $j=0$ or $j=2,\ldots, r$  together with the  element
$-\sigma$
are   representatives   of   the quasi-isolated  conjugacy  classes
in $\SL_{2r}.\sigma$. They are all isolated.

The  centralizer  of  $t(j).\sigma$ in $\SL_{2r}$
is isomorphic to $\SO_{2j}\times\Sp_{2r-2j}$. 
\end{enumerate}
In the above $\SO_0$, $\Orth_0$ and
$\SO_1$ are trivial,  $\Orth_1=\pm1$, $\SO_2$ is a rank 1 torus
and $\Orth_2$ is its extension by the
automorphism $t\mapsto t\inv$ of the torus.
\end{proposition}
\begin{proof}
By Proposition \ref{Zso=1} we already know  that  the  quasi-isolated classes 
in $\GL_r.\sigma$  are in bijection with those of $\PGL_r.\sigma$.  However
the groups $\Ts/\Tso$ are different as we shall see below.

We consider first
the groups of type $A_{2r}$, that is $\GL_{2r+1}.\sigma$,
$\PGL_{2r+1}.\sigma$ and $\SL_{2r+1}.\sigma$ with $\sigma$ quasi-central of
symplectic  type. By Proposition \ref{R(sigma)}  the root datum $R(\sigma)$
is  of type $C_r$  simply connected. We label its simple roots as
$\nnode{\alpha_1}\cdots\nnode{\alpha_2}\sbar\nnode{\alpha_{r-1}}
{\rlap{$<$}\dbar}\nnode{\alpha_r}$.
Using  \cite[Proposition 4.9]{cedric}, we
get   that  there   are  $r+1$  elements of $Y_\sigma\otimes\BQ$
parameterizing quasi-isolated classes (all isolated) which are $0$ and 
$\varpi_j^\vee/n_j$
where  $\varpi_j^\vee$ runs over the  fundamental coweights of $R(\sigma)$.
For $1\leq j \leq r-1$ we have $n_j=2$, and $n_r=1$.
Let
$\{\varpi^{\vee A}_j\mid  j=1,\ldots,2r\}$ be  the fundamental  coweights in
type  $A_{2r}$.  The  only  special  orbit  of  simple  roots  in  $A_{2r}$
corresponds   to  the   coweights  $\varpi^{\vee A}_r,\varpi^{\vee A}_{r+1}$,
hence, by Proposition \ref{Phi_sigma} we have
$\varpi_j^\vee=\pi(\varpi^{\vee A}_j)$      for     $j=1,\ldots,r-1$     and
$\varpi_r^\vee=\frac12\pi(\varpi^{\vee A}_r)$,  so  that  the
non-zero   parameters  of   the  quasi-isolated  classes  are
$\frac{\varpi^{\vee A}_j+\varpi^{\vee A}_{2r+1-j}}4$, with $i=1,\ldots,r$.
The    corresponding elements in $\GL_{2r+1}.\sigma$, $\PGL_{2r+1}.\sigma$ or
$\SL_{2r+1}.\sigma$  are  the  $t(j).\sigma$  for  $j=0,\ldots,r$. This proves
the list of representatives in (1) for odd rank.

Consider  now  $\SL_{2r}.\sigma$.  By  Proposition \ref{R(sigma)}, the root
datum  $R(\sigma)$ is of type  $B_r$ simply connected. We label the
roots in $B_r$ similarly to $C_r$ with the double bond at the end.
By \cite[Proposition
4.9]{cedric}  the parameters for quasi-isolated  classes (all isolated) are
$0$  and $\varpi_j^\vee/n_j$,  for $j=1,\ldots,r$.
We have  $n_1=1$ and  $n_j=2$ for $n>1$. Since
there is no special orbit of roots, we have $\varpi_j^\vee=\pi(\varpi^{\vee
A}_j)$  for  $j=1,\ldots,r$,  so  that  the  non-zero  parameters  for  the
quasi-isolated  classes  are  $\frac{\varpi_1^{\vee  A}+\varpi_{2r-1}^{\vee
A}}2$    and   $\frac{\varpi_j^{\vee   A}+\varpi_{2r-j}^{\vee   A}}4$   for
$j=2,\ldots,r-1$   and   $\frac{\varpi_r^{\vee   A}}2$.  The  corresponding
elements in $\SL_{2r}.\sigma$ are respectively
$\diag(-1,1,1,\ldots,1,-1).\sigma$,  and  $t(j).\sigma$  for  $j=2,\ldots,r$. 
This proves the list of representatives in (2).

Finally,  consider now the groups $\GL_{2r}.\sigma$ and $\PGL_{2r}.\sigma$.
By  Proposition \ref{R(sigma)}, the root datum $R(\sigma)$ is of type $B_r$
adjoint. If $r=1$ then $R(\sigma)$ is adjoint of type $B_1=A_1$. There is one
non-zero parameter equal to $\varpi_1^\vee=\varpi_1^{\vee A}$.
The corresponding class is not isolated. A representative is 
$\diag(1,-1).\sigma$ (which is conjugate to $t(1).\sigma$).

If   $r>1$,  by  \cite[4.E.1]{cedric},  the  non-zero  parameters  for  the
quasi-isolated  classes are $\frac{\varpi_j^\vee}2$ for $j=1,\ldots,r$, all
corresponding  to an isolated class except  when $j=1$. Lifting to type $A$
we     get     $\frac{\varpi_j^{\vee,A}+\varpi_{2r-j}^{\vee A}}4$    for
$j=1,\ldots,r$. The corresponding elements are the $t(j).\sigma$.
This proves the representatives in  (1) for even rank.

We compute now $C_\bG(t(j)\sigma)$.

If $r$ is even, the matrix $$t(j)J=
\antidiag(\underbrace{-i,\ldots,-i}_j,\underbrace{-1,\ldots,-1}_{r/2-j},
\underbrace{1,\ldots,1}_{r/2-j},\underbrace{-i,\ldots,-i}_j)$$   defines  a
bilinear  form which is the orthogonal sum of a symmetric form on the first
and  last  $j$  coordinates  and  an  alternating form on the middle $r-2j$
coordinates.  Hence  $C_\bG(t(j)\sigma)$  is  $\Orth_{2j}\times \Sp_{r-2j}$
when  $\bG=\GL_r$, is  $\SO_{2j}\times\Sp_{r-2j}$  when  $\bG=\SL_r$ and is
$(\Orth_{2j}\times\Sp_{r-2j})/\pm1$ when $\bG=\PGL_r$ (the computation is the
same  in  $\PGL$  since  if  $g$  preserves  a bilinear form up to a scalar
$\lambda$ then $\lambda^{-1/2} g$ preserves the form).

If $r$ is odd we have $$t(j)J=
\antidiag(\underbrace{-i,\ldots,-i}_j,\underbrace{-1,\ldots,-1}_{(r-1)/2-j},-1
\underbrace{1,\ldots,1}_{(r-1)/2-j},\underbrace{-i,\ldots,-i}_j),$$ hence it
defines  the orthogonal sum of  a symmetric form on  the first and last $j$
coordinates  together with  the $(r+1)/2$-th  coordinate and  of an alternating
form on the other coordinates. Hence $C_\bG(t(j)\sigma)$ is
$\Orth_{2j+1}\times  \Sp_{r-1-2j}$  when  $\bG=\GL_r$, is $\SO_{2j+1}\times
\Sp_{r-1-2j}$  when $\bG=\SL_r$ and is $(\Orth_{2j+1}\times\Sp_{r-1-2j})/\pm1
\simeq\SO_{2j+1}\times \Sp_{r-1-2j}$ when $\bG=\PGL_r$.

\end{proof}
\subsection*{Quasi-isolated classes in type $D$, $n_\sigma=2$}

The semisimple groups of type $D_n$ which admit a quasi-central automorphism
$\sigma$ with $n_\sigma=2$ are $\PSO_{2n}$, $\SO_{2n}$ and $\Spin_{2n}$, see
Proposition \ref{type of Gso}(iv). 

\begin{proposition}
If $p=2$, the points over $k$ of a semisimple group $\bG$ of type $D_n$ is
always 
the group $\SO_{2n}(k)$. The only quasi-isolated class of $\bG\cdot\sigma$ is
the class of $\sigma$, and $C_\bG(\sigma)\simeq \SO_{2n-1}(k)$.
\end{proposition}
\begin{proof} The proposition results immediately from the fact that
$R(\sigma)$ is of type $C_{n-1}$ by Proposition \ref{R(sigma)} thus has no non-central
quasi-isolated elements by \cite[5.B.2]{cedric}; then $C_\bG(\sigma)$ is
described in Proposition \ref{type of Gso}.
\end{proof}

Assume now $p\neq2$. Then
$\SO_{2n}$ is the special orthogonal group for the quadratic form
given by the antidiagonal matrix of ones. An element of the diagonal maximal
torus of $\GL_{2n}$ is in $\SO_{2n}$ if and only if its diagonal entries
are of the form $t_1,\ldots,t_n,t_n\inv,\ldots,t_1\inv$. We shall denote by
$\bT$ the maximal torus of $\SO_{2n}$, by $\diag(t_1,\ldots,t_n)\in\bT$ such
an element, and the effect of $\sigma$ on it
is to replace $t_n$ with $t_n\inv$. If $\ve_1,\ldots,\ve_n$ is the
corresponding basis of $Y(\bT)$, the simple roots are
$\alpha_1=\ve_1-\ve_2,\ldots,\alpha_{n-1}=\ve_{n-1}-\ve_n,
\alpha_n=\ve_{n-1}+\ve_n$ and the highest root is $\alpha_0=\ve_1+\ve_2$.

If $(x_1,\ldots,x_n)$ is an element of the maximal torus $\tilde\bT$
of $\Spin_{2n}$,
where the coordinates correspond to the fundamental coweights, the quotient
map $\tilde\bT\to\bT$ is given by
$(x_1,\ldots,x_n)\mapsto\diag(x_1,\frac{x_2}{x_1},\ldots,
\frac{x_{n-2}}{x_{n-3}},\frac{x_{n-1}x_n}{x_{n-2}},\frac{x_n}{x_{n-1}})$.
The kernel of this map is generated by the central element
$z:=(1,\ldots,1,-1,-1)$. The effect of $\sigma$ is to exchange $x_{n-1}$ with
$x_n$, so that $z\in[\tilde\bT,\sigma]$.

Conversely $\diag(t_1,\ldots,t_n)\in\bT$ can,
if $\overline t$ is such that $\overline t^2=t_1\cdots t_n$,
be lifted to $(t_1,t_1t_2,\ldots,t_1\cdots t_{n-2},\overline t/t_n,\overline
t)\in\tilde\bT$. 
This gives two preimages differing by $z$,
thus the conjugacy class in $\Spin_{2n}\cdot\sigma$ of $\tilde t\sigma$, where
$\tilde t$ is a lift of $\diag(t_1,\ldots,t_n)\in\bT$, is uniquely defined.

Finally the maximal torus of $\PSO_{2n}$ can be identified with $\bT/\pm1$.
\begin{proposition} Assume $p\ne2$.
Let $t_i$ be the element
$\diag(\underbrace{-1,\ldots,-1}_i,\underbrace{1,\ldots,1}_{n-i})$ of
$\SO_{2n}$, and by abuse of notation still denote by $t_i$ one of its lifts to
$\tilde\bT$.
\begin{itemize}
\item The elements $\{t_i\sigma\}_{i=0,\ldots,n-1}$ are representatives of
the $\bG$-conjugacy classes of quasi-isolated elements of $\bG\cdot\sigma$
for $\bG=\SO_{2n}$ and for $\bG=\Spin_{2n}$. These elements are isolated,
and for $\bG=\SO_{2n}$ we have $C_\bG(t_i\sigma)\simeq
\SO_{2i+1}\times\SO_{2n-2i-1}\cdot\pm 1$. For $\bG=\Spin_{2n}$ we have
$C_\bG(t_i\sigma)\simeq(\Spin_{2i+1}\times\Spin_{2n-2i-1})/(z_1,z_2)$
where $z_1$ and $z_2$ are the
generators of the center of the corresponding algebraic group (there is no
quotient to take in the extreme case $i=0$ or $i=n-1$ where there is only one
component).
\item
If $\bG=\PSO_{2n}$ then representatives of the isolated classes are $\sigma$ and
the $\{t_i\sigma\}_{i=\lfloor\frac n2\rfloor,\ldots,n-2}$.
We have $C_\bG^0(t_i\sigma)\simeq
\SO_{2i+1}\times\SO_{2n-2i-1}$. The group $C_\bG(t_i\sigma)$ is connected
unless $i=\frac{n-1}2$ in which case $C_\bG(t_i\sigma)/C_\bG^0(t_i\sigma)$ 
is of order $2$ and outer elements exchange the two components 
$C_\bG^0(t_i\sigma)$, both isomorphic to $\SO_n$.

Let $u_j$ be the element $\diag(\underbrace{-1,\ldots,-1}_j,
\underbrace{i,\ldots,i}_{n-1-2j},\underbrace{1,\ldots,1}_{j+1})$.
Representatives of the quasi-isolated classes which are not 
isolated are $\{u_j\sigma\}_{j=0,\ldots,\lfloor\frac n2\rfloor-1}$, and
$C_\bG^0(u_j\sigma)=\GL_{n-1-2j}\times\SO_{2j+1}\times\SO_{2j+1}$.
The group $C_\bG(u_j\sigma)/C_\bG^0(u_j\sigma)$ 
is of order $2$ and outer elements act on $C_\bG^0(u_j\sigma)$ 
by the outer automorphism of $\GL_{n-1-2j}$ and 
the exchange of the two components $\SO_{2j+1}$.
\end{itemize}
\end{proposition}
\begin{proof}
By Proposition \ref{R(sigma)} the root datum $R(\sigma)$ is of type $C_{n-1}$,
simply connected when $\bG$ is $\Spin_{2n}$ or $\SO_{2n}$, and adjoint when $\bG$ is
$\PSO_{2n}$.

When $R(\sigma)$ is simply connected, by \cite[4.10]{cedric}, representatives of
quasi-isolated classes for $R(\sigma)$ (which are isolated) are
$t'_i=\diag(\underbrace{-1,\ldots,-1}_i,
\underbrace{1,\ldots,1}_{n-1-i})$; here we use a $\diag$ notation to describe
the torus of $\Sp_{2n-2}$ similar to the one we used for $\SO_{2n}$. With these conventions,
the map $\bT\to\LT$ consists in forgetting the last coordinate.
The $t'_i$ lift to $\Tso\cdot\sigma$ as $t_i\sigma$, hence the
representatives in this case.

When $R(\sigma)$ is adjoint, we
use \cite[corollary 5.4]{cedric}, which contains the
elements $t'_i$ for $i=\lfloor\frac n2\rfloor,\ldots,n-1$ (taking into account
that $t'_i=-t'_i$ since we are in the adjoint group). These lift to
$\Tso\cdot\sigma$ as
the $\{t_i\sigma\}_{i=\lfloor\frac n2\rfloor,\ldots,n-2}$, except 
$t'_{n-1}=1$ which lifts as $\sigma=t_0\sigma$. 
In the same corollary the quasi-isolated elements
which are not isolated are given as the
$u'_j=\diag(\underbrace{-1,\ldots,-1}_j,
\underbrace{i,\ldots,i}_{n-1-2j},\underbrace{1,\ldots,1}_j)$ which lift as the
$u_j\sigma$.

We have thus verified the list of representatives in each case. Let us now
determine $\Ts/\Tso$. By Lemma \ref{A(Ts)}(2) this group is trivial in the simply
connected (resp.~adjoint) case, since in this case $\sigma$ permutes a basis
of $X$, the fundamental weights  (resp.~the simple roots). Remains the case of $\SO$,
where from our description we see that $\Ts=\diag(t_1,\ldots,t_{n-1},\pm 1)$
thus is not connected; $\Ts/\Tso$ is generated by $-1$.

The next step is to compute $C_\bG(t\sigma)$ for each representative $t$.
We can use Proposition \ref{cas facile} to read 
off the type of $\Phi_{t\sigma}$ from
\cite[Theorem 5.1(b)($\alpha$)]{cedric} and the values of $\Omega$ given
in \cite[example 4.10 and Corollary 5.4]{cedric}. This confirms the type
of the root system of $\Gso$ given in the statement: for $t_i$
it is of type $B_i\times B_{n-i-1}$, and for $u_j$ of type
$A_{n-2-2j}\times B_j\times B_j$.

In order to describe the center of $\Gtso$, let us describe the simple roots
of this group. The simple roots of $R(\sigma)$ are
$\beta_1:=\alpha_1,\ldots,\beta_{n-2}:=\alpha_{n-2}$ and 
$\beta_{n-1}:=\overline\alpha_n=\alpha_{n-1}+\alpha_n$. The highest root
of $R(\sigma)$ is
$\beta_0:=\alpha_1+\alpha_0=\overline{\alpha_1+\ldots+\alpha_{n-1}}$. 
According to \cite[corollary 5.4]{cedric}, 
the simple roots of $\Phi_{t_i\sigma}$ are 
$\beta_0,\ldots,\beta_{i-1},\beta_{i+1},\ldots,\beta_{n-1}$, thus
the simple roots of $\Sigma_{t_i\sigma}$ consist of the same list where we
replace $\beta_0$ by $\pi(\alpha_1+\ldots+\alpha_{n-1})$ and $\beta_{n-1}$
by $\pi(\alpha_n)$. In the basis $\ve_i$, this can be stated as follows:
if we set
$\gamma_0:=\ve_1,\gamma_1:=\ve_1-\ve_2,\ldots,\gamma_{n-2}:=\ve_{n-2}-\ve_{n-1},
\gamma_{n-1}:=\ve_{n-1}$, the simple roots of $\Sigma_{t_i\sigma}$ are
$\gamma_0,\ldots,\gamma_{i-1},\gamma_{i+1},\ldots,\gamma_{n-1}$.
This enables us to compute the center of $C_\bG^0(t_i\sigma)$ when
$\bG=\SO_{2n}$,  the kernel of the above set of roots in $\Tso=
\diag(x_1,\ldots,x_{n-1},1)$ which we find to be the trivial group. The same
computation applies in $\bT/\pm 1$. In $\tilde\bT$ the center is reduced to
$z$ which fits with the description we give.

Similarly, the simple roots of $\Phi_{u_j\sigma}$ are
$\gamma_0,\ldots,\gamma_{j-1},\allowbreak\gamma_{j+1},\ldots,
\gamma_{n-2-j},\allowbreak\gamma_{n-j},\ldots,\gamma_{n-1}$. We find that
the center of $C_\bG^0(u_j\sigma)$ consists of the
$\diag(\underbrace{1,\ldots,1}_j,\underbrace{\lambda,\ldots,\lambda}_{n-1-2j},
\underbrace{1,\ldots,1}_j)$, which justifies its description.

It  remains to  describe the  elements of  $\CA_{t\sigma}$ (see Proposition
\ref{W/W0}).   According   to   \cite[Theorem  4.4(b)($\beta$)]{cedric}  it
consists   of   the   elements   of   $\CA_{R(\sigma)}$   which   stabilize
$\Phi_{t\sigma}$.  The group  $\CA_{R(\sigma)}$ is  generated by the element
denoted by $z_n$ in \cite[table  1]{cedric}, equal to  the longest element of
the  group  generated  by  $s_{\beta_1},\ldots,s_{\beta_{n-1}}$  times  the
longest  element of the subgroup generated by 
$s_{\beta_1},\ldots,s_{\beta_{n-2}}$. This
involution maps $\beta_i$ to $\beta_{n-i-1}$; it stabilizes
$\Phi_{t_i\sigma}$ if and only  if  $i=(n-1)/2$,    and   stabilizes   always
$\Phi_{u_j\sigma}$.  In  each  case  it  induces  the  action  described in
statement on $\Sigma_{t\sigma}$.
\end{proof}
We now look at exceptional groups.
We have extended the package \CHEVIE\ (see \cite{chevie})
of \GAP\ to implement the constructions
of this paper. A coset $\bG\cdot\sigma$ is represented as a
\verb+CoxeterCoset+, and the commands \verb+QuasiIsolatedRepresentatives+
and \verb+SemisimpleCentralizer+ have been extended to work with such cosets.
\subsection*{Quasi-isolated classes in type $E_6$, $n_\sigma=2$}
The  classification is the same for $\bG$ of type $E_6$
adjoint or simply connected since
the  datum $R(\sigma)$, of type $F_4$,  has a trivial connection index.
Since $\Ts=\Tso$ by Lemma \ref{A(Ts)}(2), using Remark \ref{R(sigma) simply connected} it
follows that \qss\ elements have a connected centralizer, and that
quasi-isolated elements are isolated.

The information in the next table can be obtained by the following \CHEVIE\
commands:
\begin{verbatim}
Gs:=RootDatum("2E6sc");
l:=QuasiIsolatedRepresentatives(Gs);
c:=List(l,t->SemisimpleCentralizer(Gs,t));
List(c,x->FundamentalGroup(c.group));
\end{verbatim}

The first column of the next table lists representatives
$t\in Q^\vee\otimes\BQ/\BZ$ such that $t\sigma$ runs over the quasi-isolated
classes.
The labelling of the diagram we use is
$\nnode1\edge\nnode3\edge\vertbar42\edge\nnode5\edge\nnode6$.
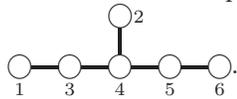
$$\begin{array}{|r|r|}
\noalign{\hrule}
\text{representatives}&C_G(t\sigma)\\
\noalign{\hrule}
0&F_4\\
\alpha^\vee_4/2&\PSp_8\\
\alpha^\vee_2/2+(\alpha^\vee_3+\alpha^\vee_5)/4&(\Spin_7\times\SL_2)/(z_1,z_2)\\
(2\alpha^\vee_2+\alpha^\vee_4)/3&(\SL_3\times \SL_3)/(z_1,z_2)\\
(3\alpha^\vee_2+2\alpha^\vee_4)/4&(\SL_4\times\SL_2)/(z_1,z_2)\\
\noalign{\hrule}
\end{array}
$$
In the above table $(z_1,z_2)$ means the element such that each component is a
generator of the center of the corresponding algebraic group.

\subsection*{Quasi-isolated classes in type $D_4$, $n_\sigma=3$}
\label{3D4}

The  classification is the same for $\bG$ adjoint or simply connected since
the  datum $R(\sigma)$, of type $G_2$,  has a trivial connection index. 
Since $\Ts=\Tso$ by Lemma \ref{A(Ts)}(2), using Remark \ref{R(sigma) simply connected} it
follows that \qss\ elements have a connected centralizer, and that
quasi-isolated elements are isolated.

The information in the next table are obtained as for the above table, using
this time \verb+"3D4sc"+ instead of \verb+"2E6sc"+.
The labelling of the diagram we use is
$\nnode{1}\edge\vertbar{3}{2}\edge\nnode{4}$
$$\begin{array}{|r|r|}
\noalign{\hrule}
\text{representatives}&C_G(t\sigma)\\
\noalign{\hrule}
0&G_2\\
\alpha_3^\vee/2&\PGL_3\\
2\alpha_3^\vee/3&(\SL_2\times\SL_2)/(z_1,z_2)\\
\noalign{\hrule}
\end{array}
$$

\vfill\eject
\printindex
\end{document}